\documentclass{amsart}

\usepackage{amssymb}
\usepackage{amsthm}
\usepackage{amsmath}
\usepackage{enumitem}
\usepackage[mathscr]{euscript}

\usepackage{cancel}
\usepackage{soul}

\usepackage[normalem]{ulem}

\usepackage[foot]{amsaddr}

\usepackage[usenames]{xcolor}

\usepackage{hyperref} 

\usepackage{tikz-cd}

\theoremstyle{plain}
\newtheorem{proposition}{Proposition}[section] 
\newtheorem{theorem}[proposition]{Theorem}
\newtheorem{lemma}[proposition]{Lemma}  

\theoremstyle{definition}
\newtheorem{example}[proposition]{Example} 
\newtheorem{definition}[proposition]{Definition}
\newtheorem{observation}[proposition]{Observation}

\theoremstyle{remark}
\newtheorem{remark}[proposition]{Remark}

\DeclareMathOperator{\Isom}{\mathsf{Isom}}

\DeclareMathOperator{\End}{End}

\DeclareMathOperator{\id}{id}

\DeclareMathOperator{\dist}{d}

\DeclareMathOperator{\Gr}{Gr}

\DeclareMathOperator{\Dc}{\mathcal{D}}
\DeclareMathOperator{\Fc}{\mathcal{F}}

\DeclareMathOperator{\Nc}{\mathcal{N}}
\DeclareMathOperator{\Oc}{\mathcal{O}}

\DeclareMathOperator{\Uc}{\mathcal{U}}

\DeclareMathOperator{\Hb}{\mathbb{H}}

\DeclareMathOperator{\Nb}{\mathbb{N}}
\DeclareMathOperator{\Pb}{\mathbb{P}}
\DeclareMathOperator{\Rb}{\mathbb{R}}

\DeclareMathOperator{\Gsf}{\mathsf{G}}
\DeclareMathOperator{\Psf}{\mathsf{P}}

\DeclareMathOperator{\PSL}{\mathsf{PSL}}

\DeclareMathOperator{\SL}{\mathsf{SL}}

\newcommand{\abs}[1]{\left|#1\right|}
\newcommand{\norm}[1]{\left\|#1\right\|}

\newcommand{\Ga}{\Gamma}
\newcommand{\ga}{\gamma}

\newcommand{\La}{\Lambda}

\newcommand{\Asf}{\mathsf{A}}

\newcommand{\Nsf}{\mathsf{N}}
\newcommand{\Ksf}{\mathsf{K}}

\newcommand{\fa}{\mathfrak{a}}
\newcommand{\mfa}{\mathfrak{a}}
\newcommand{\mfp}{\mathfrak{p}}
\newcommand{\mfg}{\mathfrak{g}}
\newcommand{\mfk}{\mathfrak{k}}

\newcommand{\Hsf}{\mathsf{H}}

\newcommand{\R}{\mathbb{R}}

\newcommand{\opp}{\mathrm{i}}

\begin{document}

\title[Vector-valued horofunction boundaries and PS-measures]{Vector-valued horofunction boundaries and Patterson--Sullivan measures}

\author[Kim]{Dongryul M. Kim}
\email{dongryul.kim97@gmail.com}
\address{Simons Laufer Mathematical Sciences Institute, USA}

\author[Zimmer]{Andrew Zimmer}
\email{amzimmer2@wisc.edu}
\address{Department of Mathematics, University of Wisconsin-Madison, USA}
\date{\today}

 \subjclass[2020]{}

\begin{abstract} In higher rank, there is a well-studied theory of Patterson--Sullivan measures supported on partial flag manifolds. However, establishing the existence and uniqueness of such  measures is a difficult question. In this paper, we develop a theory  for Patterson--Sullivan measures supported on certain vector-valued horofunction boundaries of the associated symmetric space, where existence is straightforward. We also introduce a notion of shadows for this compactification and establish a shadow lemma. For transverse groups, we prove uniqueness and ergodicity results.

\end{abstract}

\maketitle


\setcounter{tocdepth}{1}
\tableofcontents

\section{Introduction}

Throughout this paper $\Gsf$ will be a semisimple Lie group with finite center and no compact factors. We fix a Cartan decomposition $\mfg = \mfp + \mfk$ of the Lie algebra, a Cartan subspace  $\mfa \subset \mfp$, and a positive Weyl chamber $\mfa^+ \subset \mfa$. Then let $\Delta \subset \mfa^*$ denote the system of simple restricted roots corresponding  to the choice of $\mfa^+$.

Given a subset $\theta \subset \Delta$, let $\Psf_\theta < \Gsf$ denote the associated parabolic subgroup and let  $\Fc_\theta : = \Gsf /\Psf_\theta$  denote the associated partial flag manifold. There is a natural vector-valued cocycle $B_\theta^{IW} : \Gsf \times \Fc_\theta \rightarrow \mfa_\theta$ called the \emph{(partial) Iwasawa cocycle}, where  $\mfa_\theta \subset \mfa$ is the \emph{partial Cartan subspace}. This cocycle can be used to define Patterson--Sullivan measures as follows. 

\begin{definition} \label{defn:Quints definition}
Given a subgroup $\Gamma < \Gsf$, $\theta \subset \Delta$ non-empty,  a functional $\phi \in \mfa_\theta^*$, and $\delta \geq 0$, a Borel probability measure $\mu$ on $\Fc_\theta$ is a \emph{$(\Gamma, \phi, \delta)$-Patterson--Sullivan measure} if for every $\gamma \in \Gamma$ the measures $\mu$, $\gamma_*\mu$ are absolutely continuous and 
$$
\frac{d\gamma_* \mu}{d\mu}(x) = e^{-\delta \phi B_\theta^{IW}(\gamma^{-1}, x)} \quad \mu\text{-a.e.}
$$
\end{definition} 

When $\Gsf$ has rank one, the above definition (with an appropriate choice of functional) coincides with the classical Patterson--Sullivan measures introduced by Patterson \cite{Patterson1976} and Sullivan \cite{Sullivan_density}. In higher rank, the above definition is  due to Quint \cite{Quint_PS}.

For $\alpha \in \Delta$, let $\omega_\alpha$ denote the associated fundamental weight. The dual space of the  partial Cartan subspace $\mfa_\theta \subset \mfa$ can be identified with ${\rm span} \, \{ \omega_\alpha\}_{\alpha \in \theta}$. Then given $\phi \in \mfa_\theta^*= {\rm span} \, \{ \omega_\alpha\}_{\alpha \in \theta}$, the \emph{$\phi$-critical exponent} of a discrete subgroup $\Gamma < \Gsf$ is the exponential growth rate
\begin{equation} \label{eqn:exp growth rate}
\delta^\phi(\Gamma) : = \limsup_{T \rightarrow +\infty} \frac{1}{T} \log \# \left\{ \gamma \in \Gamma : \phi(\kappa(\gamma)) \leq T\right\} \in [0,+\infty]. 
\end{equation}

When $\Gsf$ has rank one, $\Delta = \{\alpha\}$ is a singleton and $\delta^{\omega_\alpha}(\Gamma)$ coincides with the classical symmetric space critical exponent. Further, there always exists a Patterson--Sullivan measure with dimension $\delta^{\omega_\alpha}(\Gamma)$. On the other hand, in higher rank, existence is much more subtle and for Zariski dense discrete subgroups the most general criteria for existence is due to Quint \cite{Quint_PS} and involves a not very easy condition to check on the growth indicator function. 

In this paper, we observe that the symmetric space $X$ associated to $\Gsf$ admits a vector-valued compactification $\overline{X}^\theta$ and there is a natural notion of Patterson--Sullivan measures on the boundary $\partial_\theta X$ of this compactification. Similar compactifications, but using Finsler metrics, appear in \cite{KL_Finsler,HSWW_horofunction,LP_horofunction_2023,Lemmens_horofunction_2023}.

We further show that the partial flag manifold $\Fc_\theta$ naturally embeds into $\partial_\theta X$ and under this embedding Patterson--Sullivan measures on $\Fc_\theta$ are sent to Patterson--Sullivan measures on $\partial_\theta X$. The boundary $\partial_\theta X$ is larger than $\Fc_\theta$, but has the advantage that existence of Patterson--Sullivan measures is straightforward to establish.  

We further show that when $\Gamma$ is sufficiently irreducible, these Patterson--Sullivan measures are part of a ``Patterson--Sullivan system,'' an abstract notion introduced in our earlier work \cite{KimZimmer1}. When $\Gamma$ is transverse, we show that these Patterson--Sullivan measures are part of a ``well-behaved Patterson--Sullivan system.''

One of the motivations for this work appears in a companion paper \cite{KimZimmer_Furstenberg}, where we combine the theory developed here with our earlier work in~\cite{KimZimmer1} to  establish new strict convexity results for variations of critical exponent and a new entropy rigidity result.

\subsection{Compactifications}

We now state the results of this paper more precisely.  Let $\Ksf < \Gsf$ denote the maximal compact subgroup with Lie algebra $\mfk$. Every element $g \in \Gsf$ can be written as $g = k e^{\kappa(g)} \ell$ for some $k, \ell \in \Ksf$ and a unique $\kappa(g) \in \mfa^+$. Then map  $\kappa : \Gsf \rightarrow \mfa^+$ is called the \emph{Cartan projection}. 

Let $X := \Gsf / \Ksf$ denote the symmetric space associated to $\Gsf$ and fix the basepoint $o : = \Ksf \in X$. For $x=go \in X$, define $b_x : X \rightarrow \mfa$ by 
$$
b_x(ho) = \kappa(h^{-1} g) - \kappa(g). 
$$
We will verify that the set $\{ b_x : x \in X\}$ is relatively compact in the space of continuous maps $X \rightarrow \mfa$ and hence we can compactify $X$ by taking the closure in this space, which we denote by $\overline{X}^{\Delta}$. We also let $\partial_{\Delta} X : = \overline{X}^{\Delta} \smallsetminus X$. 

More generally, given $\theta \subset \Delta$ non-empty there is a natural projection $\pi_\theta : \mfa \rightarrow \mfa_\theta$ we can use it to define a $\theta$-boundary $\partial_\theta X : = \{ \pi_\theta \circ \xi : \xi \in \partial_\Delta X\}$. We then show that $\partial_{\theta} X$ naturally compactifies $X$.

\begin{proposition}[see Proposition~\ref{prop:compactification} below]\label{prop:compactification intro} The space $\overline{X}^{\theta}: = X \sqcup \partial_{\theta} X$ has a topology which makes it a compactification of $X$, that is $\overline{X}^{\theta}$ is a compact metrizable space and the inclusion $X \hookrightarrow \overline{X}^{\theta}$ is a topological embedding with open dense image.

\end{proposition} 

\subsection{Patterson--Sullivan measures}

There is a well-known definition of Patterson--Sullivan measures on the  horofunction boundary of a metric space, see e.g. ~\cite{LedWang2010}. In the case of a vector-valued horofunction boundary, a choice of functional leads to a natural notion of Patterson--Sullivan measure in this compactification. 

\begin{definition}\label{defn:PS meas on partial X} 
Given a subgroup $\Gamma < \Gsf$, $\theta \subset \Delta$ non-empty,  a functional $\phi \in \mfa_\theta^*$, and $\delta \geq 0$, a Borel probability measure $\mu$ on $\partial_{\theta} X$ is a \emph{$(\Ga, \phi, \delta)$-Patterson--Sullivan measure} if for every $\gamma \in \Gamma$ the measures $\mu, \gamma_* \mu$ are absolutely continuous and 
$$
\frac{d\gamma_* \mu}{d\mu}(\xi) = e^{-\delta \phi \xi(\gamma o)} \quad \mu\text{-a.e}.
$$
\end{definition} 

Using the horofunction-like definition of these compactifications, it is possible to use Patterson's original construction to build Patterson--Sullivan measures.

\begin{proposition}[see Proposition~\ref{prop:PS meas exists} below]
   \label{prop:PS existence new boundary intro}
   Suppose $\Gamma < \Gsf$ is discrete. If $\phi \in \mfa_\theta^*$ and $\delta^\phi(\Gamma) < +\infty$, then there exists a $(\Ga, \phi, \delta^\phi(\Gamma))$-Patterson--Sullivan measure on  $\partial_{\theta} X$. 
\end{proposition}

We further show that Patterson--Sullivan measures on $\Fc_\theta$ are a special case of those on $\partial_\theta X$. 

\begin{proposition}[see Proposition~\ref{prop:embedding of partial flags} below] There is a topological embedding $\iota : \Fc_{\theta} \hookrightarrow \partial_{\theta} X$ which satisfies 
$$
\iota(x)(go) = B_{\theta}^{IW}(g^{-1}, x)
$$
for all $x \in \Fc_\theta$ and all $g \in \Gsf$. Hence the pushforward of any Patterson--Sullivan measure on $\Fc_{\theta}$ (in the sense of Definition~\ref{defn:Quints definition}) is a Patterson--Sullivan measure on $\partial_{\theta} X$ (in the sense of Definition~\ref{defn:PS meas on partial X}). 
\end{proposition}

We further show that these measures are parts of ``Patterson--Sullivan systems'' (see Section~\ref{sec:PS systems}), which were introduced in our previous work ~\cite{KimZimmer1}. This allows us to use the theory of such systems developed there and in particular implies that these Patterson--Sullivan measures satisfy a version of the shadow lemma.

In a companion paper  \cite{KimZimmer_Furstenberg}, we use the fact that these measures are parts of ``Patterson--Sullivan systems'' to apply a version of Tukia's measurable boundary rigidity theorem for such systems (established in~\cite{KimZimmer1}).

\subsection{Transverse groups} Building on work in \cite{CZZ2024}, we further establish uniqueness and ergodicity results for the class of transverse groups. Fix a non-empty $\theta \subset \Delta$. For simplicity in the introduction, we assume that $\theta$ is invariant under the opposition involution (see Equation \eqref{eqn:opp involution}), and we only consider Zariski dense subgroups. Our results in fact hold for any $\theta \subset \Delta$ and under weaker irreducibility assumptions than Zariski density (see Theorem \ref{thm:uniqueness for transverse} below).

For a Zariski dense subgroup $\Ga < \Gsf$, there exists a unique closed $\Ga$-minimal set $\La_{\theta}(\Ga) \subset \Fc_{\theta}$ called the \emph{limit set} of $\Ga$ \cite{Benoist_properties}. Then the group $\Ga$ is \emph{$\Psf_{\theta}$-transverse} if 
\begin{itemize}
   \item for any escaping $\{ \ga_n \} \subset \Ga$ and $\alpha \in \theta$,
   $$
   \alpha(\kappa(\ga_n)) \to + \infty
   $$
   and
   \item any distinct $x, y \in \La_{\theta}(\Ga)$ are transverse, i.e., the diagonal $\Gsf$-orbit
   $\Gsf \cdot (x, y) \subset \Fc_{\theta} \times \Fc_{\theta}$ is open.
\end{itemize}
The notion of transverse groups is a higher rank generalization of rank one discrete subgroups, and all Anosov and relatively Anosov groups are transverse groups. Transverse groups are sometimes called regular antipodal groups.

An important property of a $\Psf_{\theta}$-transverse group $\Ga < \Gsf$ is that the natural $\Ga$-action on the limit set $\La_{\theta}(\Ga)$ is a convergence action, see \cite[Theorem 4.16]{KLP_Anosov} or \cite[Proposition 3.3]{CZZ2026}. Hence, the notion of \emph{conical limit set} $\La_{\theta}^{\rm con}(\Ga) \subset \La_{\theta}(\Ga)$ is naturally defined.

In \cite{CZZ2024}, Canary, Zhang, and the second author established a higher rank generalization of the classical Hopf--Tsuji--Sullivan dichotomy for transverse groups which implies uniqueness of Patterson--Sullivan measure when the associated Poincar\'e series diverges. Furthermore, they also showed that such a unique measure is supported on the conical limit set. For Zariski dense groups, uniqueness was extended to measures on $\Fc_{\theta}$ by the first author, Oh, and Wang \cite{KOW_PD}.

Using this previous work and our realization of $\partial_{\theta} X$ as a part of Patterson--Sullivan system, we further extend the uniqueness of Patterson--Sullivan measure to the boundary $\partial_{\theta} X$.

\begin{theorem}[see Theorem~\ref{thm:uniqueness for transverse} below]\label{thm:uniqueness for transverse intro}  Suppose $\theta \subset \Delta$ and $\Gamma < \Gsf$ is a Zariski dense $\Psf_{\theta}$-transverse group. If  $\phi \in \mfa_\theta^*$, $\delta^{\phi}(\Ga) < + \infty$, and  
$$
\sum_{\ga \in \Ga} e^{-\delta^{\phi}(\Ga) \phi(\kappa(\ga))} = + \infty,
$$
then there is a  unique $(\Ga, \phi, \delta^{\phi}(\Ga))$-Patterson--Sullivan measure $\mu$ on $\partial_\theta X$, the $\Gamma$-action on $(\partial_\theta X, \mu)$ is ergodic, and 
   $$
   \mu(\La_{\theta}^{\rm con}(\Ga)) = 1
   $$
(in particular, $\mu$ is supported on $\Fc_\theta$).
\end{theorem}

\subsection{Outline of Paper} The first two sections of the paper are expository. In Section~\ref{sec:PS systems}  we recall  the definition and some properties of abstract Patterson--Sullivan systems, which were introduced in our earlier work \cite{KimZimmer1}. In Section~\ref{sec: notation for ss groups},  we fix the notation involving semisimple Lie groups that we will use throughout the paper. 

In Section~\ref{sec:compactifications and PS measures}, we precisely define the compactifications $\overline{X}^\theta$ and establish some basic properties. In Section~\ref{sec:shadows and cc limit set}, we define shadows in these compactifications and use them to introduce the contracting conical limit set of a discrete subgroup. In Section~\ref{sec:verifying the PS axioms}, we show that Patterson--Sullivan measures on these compactifications satisfy the axioms of abstract Patterson--Sullivan systems and prove Theorem~\ref{thm:uniqueness for transverse intro}.

\subsection*{Acknowledgements} Kim thanks the University of Wisconsin--Madison for hospitality during a visit in October 2025.
Zimmer was partially supported by a Sloan research fellowship and grants DMS-2105580 and
DMS-2452068 from the National Science Foundation. 
This material is based upon work supported by the National Science Foundation under Grant No. DMS-2424139, while the authors were in residence at the Simons Laufer Mathematical Sciences Institute in Berkeley, California, during the Spring 2026 semester.

\section{Patterson--Sullivan systems}\label{sec:PS systems} 

In this section, we recall the definition and some properties of abstract Patterson--Sullivan systems, which were introduced in our earlier work \cite{KimZimmer1}. The main idea in this previous work was to identify the key features of a group action on a probability space that allows one to extend the theory of Patterson--Sullivan measures. We note that a different framework for abstract Patterson--Sullivan-like measures was given in~\cite{BCZZ_coarse}.

Given a  compact metric space $M$, a subgroup $\Gamma < \mathsf{Homeo}(M)$, and $\kappa \geq 0$, a function $\sigma : \Ga \times M \to \mathbb{R}$ is called a \emph{$\kappa$-coarse-cocycle} if 
$$
   \left| \sigma(\ga_1 \ga_2, x) - \left( \sigma(\ga_1, \ga_2 x) + \sigma(\ga_2, x) \right) \right| \le \kappa
$$
  for any $\ga_1, \ga_2 \in \Ga$ and $x \in M$.
Given such a coarse-cocycle and $\delta \ge 0$, a Borel probability measure $\mu$ on $M$ is called \emph{coarse $(\Ga, \sigma,\delta)$-Patterson--Sullivan measure} if there exists $C\geq1$ such that for any $\ga \in \Ga$ the measures $\mu, \gamma_*\mu$ are absolutely continuous and 
   \begin{equation} \label{eqn.coarse PS meaasure intro}
   C^{-1}e^{ - \delta \sigma(\ga^{-1}, x)} \le \frac{d \ga_* \mu}{d \mu}(x) \le Ce^{- \delta \sigma(\ga^{-1}, x)} \quad \text{for } \mu\text{-a.e. } x \in M.
   \end{equation} 
When $C = 1$ and hence equality holds in Equation \eqref{eqn.coarse PS meaasure intro}, we call $\mu$ a \emph{$(\sigma,\delta)$-Patterson--Sullivan measure}.

Now we recall the definition of Patterson--Sullivan systems.

\begin{definition}\label{defn:PS systems}
A \emph{Patterson--Sullivan-system (PS-system) of dimension $\delta \ge 0$} consists of
\begin{itemize}
\item a coarse-cocycle $\sigma : \Gamma \times M \rightarrow \Rb$, 
\item a coarse $(\sigma,\delta)$-Patterson--Sullivan measure (PS-measure) $\mu$,
\item for each $\gamma \in \Gamma$, a number $\norm{\gamma}_\sigma \in \R$ called the \emph{$\sigma$-magnitude of $\gamma$}, and
\item for each $\gamma \in \Gamma$ and $R > 0$, a non-empty open set $\Oc_R(\gamma) \subset M$ called the \emph{$R$-shadow of $\gamma$}
\end{itemize}
such that:
\begin{enumerate}[label=(PS\arabic*)]
\item\label{item:coycles are bounded} For any $\ga \in \Ga$, there exists $c=c(\gamma) > 0$ such that $\abs{\sigma(\ga, x)}\leq  c(\gamma)$ for  $\mu$-a.e. $x\in M$. 
\item\label{item:almost constant on shadows} For every $R> 0$ there is a constant $C=C(R) > 0$ such that
$$
 \norm{\gamma}_\sigma - C  \leq \sigma(\gamma,x) \leq \norm{\gamma}_\sigma + C 
$$
for all $\gamma \in \Gamma$ and $\mu$-a.e. $x \in \gamma^{-1} \Oc_R(\gamma)$.
\item\label{item:empty Z intersection} If $\{\gamma_n\} \subset \Gamma$, $R_n \rightarrow +\infty$, $Z \subset M$ is compact, and $[M \smallsetminus \gamma_n^{-1}\Oc_{R_n}(\gamma_n)] \rightarrow Z$ with respect to the  Hausdorff distance, then for any $x \in Z$, there exists $g \in \Ga$ such that 
$$gx \notin Z.$$
\end{enumerate} 

We call the PS-system \emph{well-behaved} with respect to a collection
$$
\mathscr{H} := \{ \mathscr{H}(R) \subset \Ga : R \ge 0 \}
$$
 of non-increasing subsets of $\Ga$ if the following additional properties hold:
\begin{enumerate}[label=(PS\arabic*)]
\setcounter{enumi}{3}
\item\label{item:properness} $\Gamma$ is countable and for any $T > 0$, the set $\{ \gamma \in \mathscr{H}(0) : \norm{\gamma}_\sigma \leq T\}$ is finite. 

\item\label{item:baire} If $\{\gamma_n\} \subset \Gamma$, $R_n \rightarrow +\infty$, $Z \subset M$ is compact, and $[M \smallsetminus \gamma_n^{-1}\Oc_{R_n}(\gamma_n)] \rightarrow Z$ with respect to the  Hausdorff distance, then for any $h_1, \ldots, h_m \in \Ga$ and $x \in Z$, there exists $g \in \Ga$ such that
$$
g x \notin \bigcup_{i = 1}^m h_i Z.
$$

\item\label{item:shadow inclusion} If $R_1 \leq R_2$ and $\gamma \in \mathscr{H}(0)$, then $\Oc_{R_1}(\gamma) \subset \Oc_{R_2}(\gamma)$. 

\item\label{item:intersecting shadows} For any $R > 0$ there exist $C>0$ and $R'> 0$ such that: if $\alpha, \beta \in \mathscr{H}(R)$, $\norm{\alpha}_\sigma \leq \norm{\beta}_\sigma$, and $\Oc_R(\alpha) \cap \Oc_R(\beta) \neq \emptyset$, then 
$$
\Oc_R(\beta) \subset \Oc_{R'}(\alpha)
$$ 
and  
$$
\abs{\norm{\beta}_\sigma - (\norm{\alpha}_\sigma + \norm{\alpha^{-1}\beta}_\sigma)} \leq C. 
$$
\item \label{item:diam goes to zero ae} For every $R > 0$, there exists a set $M' \subset M$ of full $\mu$-measure such that 
$$
\lim_{n \rightarrow + \infty} {\rm diam} \Oc_R(\gamma_n) = 0
$$
whenever $\{\gamma_n\} \subset \mathscr{H}(R)$ is an escaping sequence and 
$$
x \in M' \cap  \bigcap_{n \ge 1} \Oc_R(\gamma_n).
$$
\end{enumerate} 
We call the collection $\mathscr{H}$ the \emph{hierarchy} of the Patterson--Sullivan system. 
\end{definition} 

The axioms above are meant to identify the key features of shadows in hyperbolic geometry. The notion of hierarchy is designed to include examples where $\Gamma$ acts on a metric space by isometries with ``contracting'' elements, but not every element is ``contracting.''  For instance, for the mapping class group action on the  Teichm\"uller space, pseudo-Anosov mapping classes are precisely ``contracting'' elements by Minsky's Contraction Theorem \cite{minsky1996quasi-projections}. In this case, the space $M$ on which Patterson--Sullivan measures are defined is the Gardiner--Masur boundary of the Teichm\"uller space. See~\cite[Section 10]{KimZimmer1}. 

PS-systems always satisfy a natural analogue of the Shadow Lemma. 

\begin{proposition}[Shadow Lemma, {\cite[Proposition 3.1]{KimZimmer1}}] \label{prop.shadowlemma} Let $(M, \Ga, \sigma, \mu)$ be a PS-system of dimension $\delta \ge 0$. For any $R > 0$ sufficiently large there exists $C=C(R) > 1$ such that 
$$
\frac{1}{C} e^{-\delta \norm{\gamma}_\sigma} \leq \mu( \Oc_R(\gamma)) \leq C e^{-\delta \norm{\gamma}_\sigma} 
$$
for all $\ga \in \Ga$.
\end{proposition} 

Recall that in hyperbolic geometry, $x \in \partial \Hb^n$ is a \emph{conical limit point} of a discrete subgroup $\Gamma < \Isom(\Hb^n)$ if there exist $R > 0$ and an escaping sequence $\{ \gamma_n\} \subset \Gamma$ such that $x$ is contained in the $R$-shadow of each $\gamma_n$. Further, the set of conical limit point has full measure if and only if the Poincar\'e series diverges, i.e.
$$
\sum_{\gamma \in \Gamma} e^{-\delta_{\Hb^n}(\Gamma) \dist(o,\gamma o)} = +\infty 
$$
by the classical Hopf--Tsuji--Sullivan dichotomy \cite{Tsuji_potential,Hopf_ergodic,Sullivan_density,AS_rational,Roblin_ergodic}. 

In a similar fashion, for well-behaved PS-systems with respect to the trivial hierarchy, divergence of the ``Poincar\'e series'' at the critical exponent implies that the set of ``conical limit points'' has positive measure. 

\begin{theorem}[{\cite[Theorem 4.1]{KimZimmer1}}] \label{thm:div => pos meas conical} Let $(M, \Ga, \sigma, \mu)$ be a PS-system of dimension $\delta \ge 0$. If $(M, \Ga, \sigma, \mu)$ is well-behaved with respect to the trivial hierarchy $\mathscr{H}(R) \equiv \Gamma$ and 
$$
\sum_{\gamma \in \Gamma} e^{-\delta \norm{\gamma}_\sigma} = +\infty, 
$$
then the set 
$$
E:=\left\{ x \in M : \exists \{\gamma_n\} \subset \Ga \text{ escaping and } R > 0 \text{ such that } x \in  \bigcap_{n \ge 1} \Oc_R(\gamma_n) \right\} 
$$
has positive $\mu$-measure. 
\end{theorem} 

\begin{remark} In many particular examples, one can prove that if $g \in \Gamma$ and $R > 0$, then there exists $R' > 0$ such that 
$$
g \Oc_R(\gamma) \subset \Oc_{R'}(g \gamma)
$$
for all $\gamma \in \Gamma$. In this case, the set $E$ is Theorem~\ref{thm:div => pos meas conical} is $\Gamma$-invariant and further has full $\mu$-measure. Indeed, if $\mu(E) < 1$, then the $\Gamma$-invariance implies that $\mu' := \frac{1}{\mu(M \smallsetminus E)} \mu( \cdot \cap (M \smallsetminus E))$ is also a PS-measure, but then applying Theorem~\ref{thm:div => pos meas conical} to this measure shows that $\mu'(E) > 0$ which is impossible. 

\end{remark}

\section{Notations for semisimple Lie groups}\label{sec: notation for ss groups}

In this section we fix the notation involving semisimple Lie groups that we will use throughout the paper. Of particular importance for our arguments are the linear representations fixed in Section~\ref{sec:linear representations}. 

Recall from the introduction that $\Gsf$ is a semisimple Lie group with finite center and no compact factors, $\mfg = \mfp + \mfk$ is a fixed Cartan decomposition of the Lie algebra, $\mfa \subset \mfp$ is a fixed Cartan subspace, and $\mfa^+ \subset \mfa$ is a fixed positive Weyl chamber. We use $\Sigma \subset \mfa^*$ to denote the set of restricted roots and use $\Delta \subset \mfa^*$ to denote the system of simple restricted roots corresponding  to the choice of $\mfa^+$. Then 
$$
\mfg = \mfg_0 \oplus \bigoplus_{\alpha \in \Sigma} \mfg_\alpha
$$
where 
$$
\mfg_\alpha = \{ X \in \mfg : [H,X] = \alpha(H)X \text{ for all } H \in \mfa\}.
$$
Let $\Sigma^+$ (resp. $\Sigma^-$) denote the restricted roots which are non-negative (respectively non-positive) linear combinations of elements of $\Delta$.

Recall that $\kappa : \Gsf \rightarrow \mfa^+$ denotes the Cartan projection. We fix a representative $w_0 \in \Ksf$ of the longest Weyl element which is of order $2$. Let $\opp := - {\rm Ad}_{w_0} : \fa \to \fa$ denote the \emph{opposition involution}. This map has the property that 
\begin{equation} \label{eqn:opp involution}
\kappa(g^{-1}) = \opp \kappa(g) \quad \text{for all} \quad g \in \Gsf. 
\end{equation}
The adjoint $\opp^*$ of the opposite involution preserves the set of simple roots and for a subset $\theta \subset \Delta$, we define 
$$
\opp^* \theta := \{ \opp^*\alpha : \alpha \in \theta \}.
$$

\subsection{Parabolic subgroups and flag manifolds} Given a non-empty $\theta \subset \Delta$, the associated parabolic subgroup $\Psf_\theta$ is the stabilizer under the adjoint action of the Lie algebra 
$$
\mathfrak{u}_\theta^+ : = \bigoplus_{\alpha \in \Sigma_\theta^+} \mfg_\alpha
$$
where $\Sigma_\theta^+ : = \Sigma^+ \smallsetminus {\rm span}(\Delta \smallsetminus \theta)$. We also set
$ \Asf := \exp \fa$ and $\Asf^+ := \exp \fa^+$, and denote by $\Nsf < \Psf_{\Delta}$ the unipotent radical of $\Psf_{\Delta}$.

The \emph{Furstenberg boundary} and general \emph{$\theta$-boundary} are the quotient spaces 
$$
\Fc_{\Delta} := \Gsf / \Psf_{\Delta} \quad \text{and} \quad \Fc_{\theta} := \Gsf / \Psf_{\theta}.
$$
Two elements $x \in \Fc_\theta$ and $y \in \Fc_{\opp^* \theta}$ are \emph{transverse} if there exists $g \in \Gsf$ such that 
$$
x = g \Psf_\theta \quad \text{and} \quad y = g w_0 \Psf_{\opp^* \theta},
$$
equivalently $(x,y)$ is contained in the unique open $\Gsf$-orbit in $\Fc_\theta \times \Fc_{\opp^* \theta}$.

\subsection{Projection to the flag manifold}\label{sec:proj to flag manifolds}

For $g \in \Gsf$ with $\min_{\alpha \in \theta} \alpha(\kappa(g)) > 0$, we define 
$$
U_{\theta}(g) := k \Psf_\theta \in \Fc_{\theta}
$$
where $g$ has Cartan decomposition $g = k a \ell \in \Ksf \Asf^+ \Ksf$ (the condition on the roots implies that $U_{\theta}(g)$ is well-defined). 

\begin{observation}\label{obs:projection of inverse} If $g \in \Gsf$ has Cartan decomposition $g = k a \ell \in \Ksf \Asf^+ \Ksf$ and $\min_{\alpha \in \theta} \alpha(\kappa(g)) > 0$, then $\min_{\alpha \in \opp^* \theta} \alpha(\kappa(g^{-1})) > 0$ and
$$
U_{\opp^* \theta}(g^{-1}) = \ell^{-1} w_0 \Psf_{\opp^* \theta}. 
$$
\end{observation} 

\begin{proof} Notice that $g^{-1}$ has Cartan decomposition $g^{-1} = \ell^{-1} w_0 (w_0 a w_0) w_0 k^{-1}$. \end{proof} 

These projection maps have the following dynamical behavior (for a proof see for instance \cite[Section 4]{KLP_Anosov} or \cite[Proposition 2.3]{CZZ2024}). 

\begin{proposition}\label{prop:NS dynamics} If $\{g_n\} \subset \Gsf$, $x^+ \in \Fc_\theta$, and $x^- \in \Fc_{\opp^* \theta}$, then the following are equivalent: 
\begin{enumerate}
\item $g_n x \rightarrow x^+$ for all $x \in \Fc_\theta$ transverse to $x^-$ and the convergence is uniform on compact subsets.
\item $\min_{\alpha \in \theta} \alpha(\kappa(g_n))\rightarrow +\infty$, $U_\theta(g_n) \rightarrow x^+$, and $U_{\opp^* \theta}(g_n^{-1}) \rightarrow x^-$. 
\end{enumerate} 

\end{proposition}

\subsection{The partial Iwasawa cocycle} 

The \emph{Iwasawa cocycle} $B_{\Delta}^{IW} : \Gsf \times \Fc_{\Delta} \to \fa$ is defined as follows: for $g \in \Gsf$ and $x \in \Fc_{\Delta}$, $B_{\Delta}^{IW}(g, x) \in \fa$ is the unique element such that
$$
gk \in \Ksf ( \exp B_{\Delta}^{IW}(g, x)) \Nsf
$$
for $k \in \Ksf$ such that $k \Psf_{\Delta} = x$ in $\Fc_{\Delta}$. 

For general $\theta \subset \Delta$, let 
$$
\mfa_\theta : = \{ H \in \mfa : \alpha(H) = 0 \text{ for all } \alpha \notin \theta\}.
$$
For $\alpha \in \Delta$, let $\omega_\alpha$ denote the (restricted) fundamental weight associated to $\alpha$.
Then $\{ \omega_\alpha|_{\mfa_\theta}\}_{\alpha \in \theta}$ is a basis for $\mfa_\theta^*$ and so there exists a unique projection 
\begin{equation*}
\pi_{\theta} : \fa \to \fa_{\theta}
\end{equation*}
satisfying 
\begin{equation}\label{eqn:defn of pi_theta to mfa_theta} 
\omega_\alpha \pi_\theta(H) = \omega_\alpha(H)
\end{equation} 
for all $H \in \mfa$ and $\alpha \in \theta$.

The \emph{partial Iwasawa cocycle} $B_{\theta}^{IW} : \Gsf \times \Fc_{\theta} \to \fa_{\theta}$ is defined as 
\begin{equation}\label{eqn:defn of BIW} 
B_{\theta}^{IW}(g, x) := \pi_{\theta} B_{\Delta}^{IW}(g, \tilde x)
\end{equation} 
for any $\tilde x \in \Fc_{\Delta}$ that projects to $x \in \Fc_{\theta}$ under the canonical projection $\Fc_{\Delta} \to \Fc_{\theta}$. The above definition is independent of the choice of $\tilde x$ and defines a cocycle \cite[Lemma 6.1]{Quint_PS}. 

Recall from the introduction that the partial Iwasawa cocycle can be used to define a notion of Patterson--Sullivan measures on the partial flag manifold $\Fc_\theta$ (see Definition~\ref{defn:Quints definition}).

\subsection{Linear Representations}\label{sec:linear representations}

Given a $d$-dimensional vector space $V$ endowed with an inner product and $g \in \SL(V)$, we let 
$$
\sigma_1(g)  \geq \cdots \geq \sigma_d(g)
$$
denote the singular values of $g$ with respect to the inner product and let $\norm{g} = \sigma_1(g)$ denote the operator norm.

 Throughout the paper,  for each $\alpha \in \Delta$ we fix an irreducible representation $\Phi_{\alpha} : \Gsf \to \SL(V_{\alpha})$ and a $\Phi_\alpha(\Ksf)$-invariant inner product on $V_\alpha$ with the following properties: 
\begin{enumerate}[label=(R\arabic*)]
\item\label{item:SVs of Tits repn} 
There exists $N_{\alpha} \in \Nb$ 
such that if $g \in \Gsf$, then 
$$
\log \norm{\Phi_\alpha(g)}= N_{\alpha} \omega_\alpha(\kappa(g)) \quad \text{and}\quad \log \frac{\sigma_1(\Phi_{\alpha}(g))}{\sigma_2(\Phi_{\alpha}(g))} = \alpha(\kappa(g)).
$$
\item\label{item:splitting for Tits repn}  There exists a $\Phi_\alpha(\Asf)$-invariant orthogonal splitting $V_\alpha = V_\alpha^+ \oplus V_\alpha^-$ such that $\dim V_\alpha^+ = 1$. Moreover, if  $H \in \mfa$ and  $v \in V_\alpha^+$, then 
$$
\Phi_\alpha(e^H)v = e^{N_{\alpha} \omega_\alpha(H)} v.
$$

\item\label{item:boundary maps of Tits repn} There exist $\Phi_\alpha$-equivariant boundary maps $\zeta_\alpha : \Fc_\alpha \rightarrow \Pb(V_\alpha)$ and $\zeta_\alpha^* : \Fc_{\opp^* \alpha} \rightarrow \Gr_{\dim V_\alpha-1}(V_\alpha)$ such that:
\begin{enumerate}
\item $\zeta_\alpha(\Psf_{\alpha}) = V_\alpha^+$ and $\zeta_\alpha^*(w_0 \Psf_{\opp^* \alpha } ) = V_\alpha^-$.
\item $x \in \Fc_\alpha$ and $y \in \Fc_{\opp^* \alpha}$ are transverse if and only if $\zeta_\alpha(x)$ and $\zeta_\alpha^*(y)$ are transverse. 
\end{enumerate} 

\end{enumerate} 

\begin{remark} 

   Such representations exist due to Tits \cite[Theorem 7.2]{Tits_representations}. Indeed, Tits proved the first claim in Property \ref{item:SVs of Tits repn}. For a proof of the second assertion in Property~\ref{item:SVs of Tits repn} and Property~\ref{item:splitting for Tits repn}, see for instance \cite[Lemma 2.13]{Smilga_proper} and \cite[Sections 6.8, 6.9]{BQ_book}. For a proof of Property~\ref{item:boundary maps of Tits repn}, see for instance ~\cite[Section 3]{GGKW2017}.

\end{remark}

\begin{remark} 
   We abuse notation and when $\alpha \in \theta$, also often use $\zeta_\alpha$ to also denote the map $\Fc_\theta \rightarrow \Pb(V_\alpha)$ obtained by precomposing $\zeta_\alpha: \Fc_\alpha \rightarrow \Pb(V_\alpha)$ with the natural projection $\Fc_\theta \rightarrow \Fc_\alpha$. Likewise, we also use $\zeta_{\opp ^*\alpha}^*$ to denote the analogous map defined on $\Fc_{\opp^* \theta}$. 
\end{remark}

The following lemma relates the projections to the flag manifolds introduced in Section~\ref{sec:proj to flag manifolds} to these representations. 

\begin{lemma}\label{lem:rank one limits in Valpha} Fix $\alpha \in \Delta$ and  assume $\{g_n\} \subset \Gsf$ is such that  $\alpha(\kappa(g_n)) \rightarrow +\infty$, $U_\alpha(g_n) \rightarrow x$,  and $U_{\opp^* \alpha }(g_n^{-1}) \rightarrow y$. Then any limit point of 
$$
\frac{1}{\norm{\Phi_\alpha(g_n)}} \Phi_\alpha(g_n) \quad \text{in } \End(V_{\alpha})
$$
has image $\zeta_\alpha(x)$ and kernel $\zeta_\alpha^*(y)$. 
\end{lemma} 

\begin{proof} Suppose that $\frac{1}{\norm{\Phi_\alpha(g_n)}} \Phi_\alpha(g_n) \rightarrow T$. Fix a Cartan decomposition $g_n = k_n a_n \ell_n \in \Ksf \Asf^+ \Ksf$. Passing to subsequences we can suppose that $k_n \rightarrow k$, $\ell_n \rightarrow \ell$, and $\frac{1}{\norm{\Phi_\alpha(a_n)}} \Phi_\alpha(a_n) \rightarrow S$. Then $x = k \Psf_{\alpha}$   and $T = \Phi_\alpha(k) S \Phi_\alpha(\ell)$. Further,  Observation~\ref{obs:projection of inverse} implies that $y = \ell^{-1} w_0 \Psf_{\opp^* \alpha }$.

By Properties~\ref{item:SVs of Tits repn} and~\ref{item:splitting for Tits repn}, $S$ has image $V_\alpha^+$ and kernel $V_\alpha^-$. Then by Property~\ref{item:boundary maps of Tits repn}, $T$ has image 
$$
\Phi_\alpha(k)V_\alpha^+ = \zeta_\alpha(k \Psf_{\alpha}) = \zeta_\alpha(x)
$$
and kernel 
\begin{equation*}
\Phi_\alpha(\ell)^{-1}V_\alpha^- = \zeta_\alpha^*(\ell^{-1} w_0 \Psf_{\opp^* \alpha }) = \zeta_\alpha^*(y). \qedhere
\end{equation*}
\end{proof}

\subsection{Irreducible actions} 

A subgroup $\Hsf < \SL(V)$ is called \emph{irreducible} if there is no non-trivial and proper subspace of $V$ invariant under $\Hsf$, and is called \emph{strongly irreducible} if any finite index subgroup of $\Hsf$ is irreducible. We transfer these notions to $\Gsf$ using the $\Phi_{\alpha}$'s.

\begin{definition} A subgroup $\Ga < \Gsf$ is \emph{$(\Phi_{\alpha})_{\alpha \in \theta}$-irreducible} if $\Phi_{\alpha}(\Ga) < \SL(V_{\alpha})$ is irreducible for all $\alpha \in \theta$, and \emph{strongly  $(\Phi_{\alpha})_{\alpha \in \theta}$-irreducible}  if any finite index subgroup of $\Ga$ is  $(\Phi_{\alpha})_{\alpha \in \theta}$-irreducible. 
\end{definition} 

\begin{remark}  \label{remark:Zdense irreducible}
   Notice that a Zariski dense subgroup is strongly  $(\Phi_{\alpha})_{\alpha \in \Delta}$-irreducible. \end{remark} 

We will use the following observation several times. 
 
\begin{lemma} \label{lem:rho irr implies flag irr}
Suppose $\Ga < \Gsf$ is strongly  $(\Phi_{\alpha})_{\alpha \in \theta}$-irreducible. If 
\begin{itemize}
\item $\alpha_1,\dots, \alpha_m$ are (possibly non-distinct) elements of $\theta$, 
\item $v_i \in V_{\alpha_i} \smallsetminus\{0\}$ for $i=1,\dots,m$, and 
\item $W_i \subset V_{\alpha_i}$ is a proper linear subspace for $i=1,\dots, m$, 
\end{itemize} 
then there exists $\gamma \in \Gamma$ with 
$$
\Phi_{\alpha_i}(\gamma) v_i \notin W_i
$$
for $i=1,\dots, m$. 
\end{lemma}

\begin{proof}
Let $\Hsf < \Gsf$ be the Zariski closure of $\Ga$ in $\Gsf$ and let $\Hsf^0 < \Hsf$ be the identity component. Since $\Ga$ is strongly $(\Phi_{\alpha})_{\alpha \in \theta}$-irreducible, $\Hsf^0$ is $(\Phi_{\alpha})_{\alpha \in \theta}$-irreducible as well. Then for each $1 \le i \le m$, the set
$$
\Oc_i := \{ h \in \Hsf^0 : \Phi_{\alpha_i}(h) v_i \notin W_i \}
$$
is non-empty and Zariski open. Since $\Ga \cap \Hsf^0$ is Zariski dense in $\Hsf^0$, there exists
$$
\ga \in \Ga \cap \bigcap_{i = 1}^m \Oc_i,
$$
which satisfies the desired properties.
\end{proof}

\subsection{Limit sets and transverse groups}\label{sec:transverse groups}

In this section we recall the definition a transverse group.

Given a discrete subgroup $\Gamma < \Gsf$, a point $x \in \Fc_{\theta}$ is a \emph{limit point} of $\Ga$ if there exists an escaping sequence $\{ \ga_n \} \subset \Ga$ with $\min_{\alpha \in \theta} \alpha(\kappa(\gamma_n)) \to + \infty$ and  $U_{\theta}(\ga_n) \rightarrow x$. The \emph{limit set of $\Gamma$}, denoted by
$$
\La_{\theta}(\Ga) \subset \Fc_{\theta},
$$
is the set of all limit points of $\Ga$. Proposition~\ref{prop:NS dynamics} implies that the points in $\Lambda_\theta(\Gamma)$ are exactly the points $x^+ \in \Fc_\theta$ where there exists a sequence $\{\gamma_n\} \subset \Gamma$ and a non-empty open set $\Uc \subset \Fc_\theta$ such that $\gamma_n x \rightarrow x^+$ for all $x \in \Uc$, uniformly on compact subsets.

A discrete subgroup  $\Ga < \Gsf$ is \emph{$\Psf_{\theta}$-transverse} if $\min_{\alpha \in \theta} \alpha(\kappa(\gamma_n)) \to + \infty$ for any sequence $\{ \ga_n \} \subset \Ga$ of distinct elements and any two distinct points in $\La_{\theta \cup \opp^* \theta}(\Ga)$ are transverse.

Sometimes the definition of transverse group includes the assumption that $\theta$ is symmetric (i.e., $\theta = \opp^* \theta)$. However, as the next observation demonstrates, this results in no loss of generality. 

\begin{observation}\label{obs:symmetry theta} $\Gamma < \Gsf$ is $\Psf_{\theta}$-transverse if and only if $\Gamma$ is $\Psf_{\theta \cup \opp^*\theta}$-transverse. Moreover, in this case the the projection $\Fc_{\theta \cup \opp^* \theta} \rightarrow \Fc_\theta$ induces a homeomorphism $\Lambda_{\theta \cup \opp^* \theta}(\Gamma) \rightarrow \Lambda_\theta(\Gamma)$.
\end{observation}

A $\Psf_{\theta}$-transverse group is called \emph{non-elementary} if $\# \La_{\theta}(\Ga) \ge 3$, in which case the natural $\Ga$-action on $\La_{\theta}(\Ga)$ is a minimal convergence action  and $\# \La_{\theta}(\Ga) = + \infty$, see \cite[Theorem 4.16]{KLP_Anosov} or \cite[Proposition 3.3]{CZZ2026}.

Two well-studied classes of transverse groups are Anosov groups and relatively Anosov groups. A $\Psf_{\theta}$-transverse group is  \emph{$\Psf_{\theta}$-Anosov} if the $\Ga$-action on $\La_{\theta}(\Ga)$ is a uniform convergence action, equivalently $\Gamma$ is word hyperbolic as an abstract group and there exists an equivariant homeomorphism from the Gromov boundary to the limit set $\La_{\theta}(\Ga)$ \cite{bowditch1998a-topological}. Likewise, a $\Psf_{\theta}$-transverse group is  \emph{$\Psf_{\theta}$-relatively Anosov} if the  $\Ga$-action on $\La_{\theta}(\Ga)$ is geometrically finite, equivalently if $\Gamma$ has the structure of a relatively hyperbolic group and there exists an equivariant homeomorphism from the associated Bowditch boundary to the limit set $\La_{\theta}(\Ga)$ \cite{Yaman2004topological}.


\section{Compactifications and Patterson--Sullivan measures}\label{sec:compactifications and PS measures}


In this section, we introduce vector-valued horofunction compactifications of the symmetric space $X = \Gsf/\Ksf$ associated to $\Gsf$, using Cartan projections. It turns out that they contain partial flag manifolds, and we also consider Patterson--Sullivan measures there.

Fix the basepoint $o:=\Ksf \in X$. The symmetric space distance is given by 
$$
\dist_X(go, ho) = \norm{\kappa(g^{-1}h)}
$$
where $\norm{\cdot}$ is some norm on $\mfa$. 
For $x = go$, define  $b_x : X \rightarrow \fa$ by 
$$
b_x(ho) = \kappa(h^{-1}g) - \kappa(g). 
$$

\begin{lemma} \label{lem:Lipschitz}
The maps $\{ b_x : x \in X\}$ are uniformly Lipschitz. 
\end{lemma}

\begin{proof}
Fix $x = g o\in X$. Let $h_1o, h_2o \in X$. Then for each $\alpha \in \Delta$, it follows from Property \ref{item:SVs of Tits repn} that
$$\begin{aligned}
N_{\alpha}\omega_{\alpha}\left( b_x(h_1o) - b_x(h_2o)\right) & = \log \norm{\Phi_{\alpha}(h_1^{-1}g)} - \log \norm{\Phi_{\alpha}(h_2^{-1}g)} \\
& \le \log \norm{ \Phi_{\alpha}(h_1^{-1}h_2)} = N_{\alpha} \omega_{\alpha} (\kappa(h_1^{-1}h_2)).
\end{aligned}
$$
And similarly, 
$$
\omega_{\alpha}\left( b_x(h_2o) - b_x(h_1o)\right) \le \omega_{\alpha}(\kappa(h_2^{-1}h_1)).
$$
This implies
$$
\sum_{\alpha \in \Delta} \abs{\omega_{\alpha}\left( b_x(h_2o) - b_x(h_1o)\right)} \le 2 \sum_{\alpha \in \Delta} \abs{\omega_{\alpha}(\kappa({h_1}^{-1}h_2))} .
$$
Since $\sum_{\alpha \in \Delta} \abs{\omega_{\alpha}(\cdot)}$ is a norm on $\fa$ and 
$$
\dist_X(h_1o,h_2o) = \norm{\kappa(h_1^{-1}h_2)}, 
$$
 this finishes the proof.
\end{proof}

For non-empty $\theta \subset \Delta$,  let $\pi_\theta : \mfa \rightarrow \mfa_\theta$ be the projection satisfying Equation~\eqref{eqn:defn of pi_theta to mfa_theta}  and then let  $\partial_{\theta} X$ be the set of functions $\xi : X \rightarrow \mfa_\theta$ where there exists an escaping sequence $\{x_n\} \subset X$ with $\pi_\theta b_{x_n} \rightarrow \xi$ in the compact-open topology. Lemma \ref{lem:Lipschitz}, together with the separability of $X$, implies that $\partial_\theta X$ is compact in the compact-open topology. Further, $\Gsf$  acts on $\partial_{\theta} X$  by 
$$
g\cdot \xi = \xi \circ g^{-1} - \xi(g^{-1}o).
$$

The next proposition shows that $\partial_\theta X$ can be used to compactify $X$. 

\begin{proposition}\label{prop:compactification} The space $\overline{X}^{\theta}: = X \sqcup \partial_{\theta} X$ has a topology which makes it a compactification of $X$, that is $\overline{X}^{\theta}$ is a compact metrizable space and the inclusion $X \hookrightarrow \overline{X}^{\theta}$ is a topological embedding with open dense image. Moreover with respect to this topology:
\begin{enumerate}
\item  $\{x_n\} \subset X$ converges to $\xi \in \partial_{\theta} X$ if and only if $\dist_X(o,x_n) \rightarrow +\infty$ and $\pi_\theta b_{x_n} \rightarrow \xi$ in the compact-open topology.
\item The $\Gsf$-action on $\overline{X}^{\theta}$ is continuous.
\item The function  $B_\theta : \Gsf \times \overline{X}^{\theta} \rightarrow \fa$ defined by 
$$
B_\theta(g,x) = \begin{cases} \pi_{\theta} b_x(g^{-1} o) & \text{ if } x \in X \\ x(g^{-1} o) & \text{ if } x \in \partial_{\theta} X \end{cases}
$$
is continuous. 
\end{enumerate} 
\end{proposition} 

\begin{remark} 
Notice that the function $B_\theta$ is a linear cocycle: 
$$
B_\theta(g_1g_2,x) = B_\theta(g_1, g_2 x) + B_\theta(g_2, x)
$$
for all $g_1,g_2 \in \Gsf$ and $x \in \overline{X}^\theta$.   
\end{remark}

\begin{proof} For two Lipschitz functions $\xi_1,\xi_2 : \Gsf \rightarrow \mfa_\theta$ define 
$$
\dist_0(\xi_1,\xi_2) := \sum_{n \ge 1} \frac{1}{2^n} \max_{x \in \overline{B_X(o, n)}} \norm{\xi_1(x)-\xi_2(x)}
$$
where $B_X(o, n) \subset X$ is the $n$-ball centered at $o$.
Also define $h : X \rightarrow (0, +\infty)$ by $h(x) := \frac{1}{1+\dist_X(o,x)}$. Then define a metric $\dist$ on $\overline{X}^{\theta}$ by 
\begin{align*}
\dist(x,y) & :=  \min\left\{ \dist_X(x,y), h(x)+h(y) \right\} + \dist_0(b_x,b_y)   \quad \text{if} \quad x,y \in X, \\
\dist(x,\xi) & := h(x) + \dist_0(b_x, \xi)   \quad \text{if} \quad x\in X, \ \xi \in \partial_\theta X, \\
\dist(\xi_1,\xi_2) & := \dist_0(\xi_1,\xi_2)   \quad \text{if} \quad \xi_1,\xi_2 \in \partial_\theta X.
\end{align*}
Following \cite[Section 3]{Mandelkern_compactification}, together with the fact that $\abs{h(x) - h(y)} \le \dist_X(x, y)$ for $x, y \in X$, one can check that $\dist(\cdot, \cdot)$ is indeed a metric on $\overline{X}^{\theta}$, and the topology induced by this metric has all the desired properties. 
\end{proof} 

Patterson--Sullivan measures on $\partial_{\theta} X$ can naturally be defined as follows. 

\begin{definition}\label{defn: PS on X theta}
Given a subgroup $\Gamma < \Gsf$, $\theta \subset \Delta$ non-empty,  a functional $\phi \in \mfa_\theta^*$, and $\delta \geq 0$,  a Borel probability measure $\mu$ on $\partial_{\theta} X$ is a \emph{coarse $(\Gamma, \phi, \delta)$-Patterson--Sullivan measure} if  there exists $C \ge 1$ such that for every $\gamma \in \Gamma$ the measures $\mu$, $\gamma_*\mu$ are absolutely continuous and 
$$
C^{-1}  e^{-\delta \phi  \xi(\ga o)} \le \frac{d\gamma_* \mu}{d\mu}(\xi) \le  C e^{-\delta \phi \xi(\ga o)} \quad \mu\text{-a.e.}
$$
We call $\mu$ a \emph{$(\Ga, \phi, \delta)$-Patterson--Sullivan measure} if $C = 1$. 
\end{definition} 

Using Patterson's original construction for Fuchsian groups, we can prove the following existence result.

\begin{proposition} \label{prop:PS meas exists}
   If $\Gamma < \Gsf$ is discrete, $\phi \in \mfa_\theta^*$, and $\delta^\phi(\Gamma) < +\infty$, then there exists a $(\Gamma, \phi, \delta^\phi(\Gamma))$-Patterson--Sullivan measure on $\partial_\theta X$. 
\end{proposition} 

\begin{proof} For $\gamma \in \Gamma$ fixed, the function $f_\gamma:=\phi B_\theta(\gamma, \cdot) : \overline{X}^\theta \rightarrow \Rb$ is continuous. Hence we can follow Patterson's original argument for constructing Patterson--Sullivan measures for Fuchsian groups. In particular, by ~\cite[Lemma 3.1]{Patterson1976}, there exists a continuous non-decreasing function $h : \Rb \to \Rb^+$
   such that:
\begin{enumerate}
\item The series $\sum_{\gamma \in \Gamma} h( \phi(\kappa(\gamma))) e^{-s\phi(\kappa(\gamma))}$ diverges at $s = \delta^\phi(\Gamma)$ and converges for $s >  \delta^\phi(\Gamma)$. 
\item For any $\epsilon > 0$ there exists $t_0 > 0$ such that if $t > t_0$ and $s > 1$, then $h(st) \leq s^\epsilon h(t)$. 
\end{enumerate}
(In the case when $\sum_{\gamma \in \Gamma}  e^{-\delta^\phi(\Gamma)\phi(\kappa(\gamma))}=+\infty$, we can take $h \equiv 1$.)

For $s > \delta^\phi(\Gamma)$, consider the probability measure 
$$
\mu_s : = \frac{1}{\sum_{\gamma \in \Gamma} h( \phi(\kappa(\gamma))) e^{-s\phi(\kappa(\gamma))}} \sum_{\gamma \in \Gamma} h( \phi(\kappa(\gamma))) e^{-s\phi(\kappa(\gamma))} \Dc_{\gamma o}
$$
on $ \overline{X}^\theta$, where $\Dc_{\gamma o}$ is the Dirac mass at $\gamma o$. Next, fix $s_n \searrow \delta^\phi(\Gamma)$ such that $\mu_{s_n}$ converges to a probability measure $\mu$ in the weak-$*$ topology. Then using the continuity of $B_\theta(\gamma^{-1}, \cdot)$ on $\overline{X}^\theta$ it is straightforward to show that $\mu$ is a $(\Gamma, \phi, \delta^\phi(\Gamma))$-Patterson--Sullivan measure.
\end{proof}

Next we provide a formula for elements in $\partial_\theta X$ in terms of the linear representations introduced in Section~\ref{sec:linear representations}.

\begin{lemma} \label{lem:Titsexpress}
If $g_n o \rightarrow \xi$ in $\overline{X}^\theta$ and $\frac{ \Phi_{\alpha}(g_n)}{\norm{\Phi_{\alpha}(g_n)}} \rightarrow T_\alpha$ in $\End(V_\alpha)$ for all $\alpha \in \theta$, then 
$$
N_{\alpha} \omega_{\alpha} \xi(ho) = \log \norm{\Phi_\alpha(h^{-1})T_\alpha}
$$
for all $h \in \Gsf$ and all $\alpha \in \theta$. In particular, 
$$
\omega_\alpha\xi(h^{-1}o) \leq \omega_\alpha \kappa(h)
$$
for all $h \in \Gsf$ and all $\alpha \in \theta$. 
\end{lemma}

\begin{proof}
Let $g \in \Gsf$ and recall that $b_{go}(ho) = \kappa(h^{-1} g) - \kappa(g)$. For each $\alpha \in \theta$, Property \ref{item:SVs of Tits repn} implies that
$$
N_{\alpha} \omega_{\alpha} \kappa(h^{-1} g) = \log \norm{ \Phi_{\alpha}(h^{-1} g)} \quad \text{and} \quad 
N_{\alpha} \omega_{\alpha} \kappa(g) = \log \norm{ \Phi_{\alpha}(g)}.
$$
Hence,
$$
N_{\alpha} \omega_{\alpha}b_{go}(ho) = \log \norm{ \Phi_{\alpha}(h^{-1}) \frac{ \Phi_{\alpha}(g)}{\norm{\Phi_{\alpha}(g)}}}.
$$
The first claim then follows. For the ``in particular'' part, since $\norm{T_\alpha } =1$, 
\begin{equation*}
\omega_\alpha \kappa(h)  = \frac{1}{N_{\alpha}} \log \norm{\Phi_\alpha(h)} \geq   \frac{1}{N_{\alpha}}\log \norm{\Phi_\alpha(h)T_\alpha} = \omega_\alpha\xi(h^{-1}o). \qedhere
\end{equation*}
\end{proof}

Recall that $B_\theta^{IW} : \Gsf \times \Fc_\theta \rightarrow~\mfa_\theta$ denotes the partial Iwasawa cocyle. Using this cocycle, we show that $\Fc_\theta$ embeds into $\partial_\theta X$.

\begin{proposition} \label{prop:embedding of partial flags}
There is a topological embedding  $\iota : \Fc_\theta \rightarrow \partial_\theta X$ that satisfies 
\begin{equation}\label{eqn:definition of iota}
\iota(x)(ho) = B_\theta^{IW}(h^{-1},x)
\end{equation}
for all $x \in \Fc_\theta$ and $h \in \Gsf$. 

Moreover:
\begin{enumerate}
\item If a sequence $\{g_n\} \subset \Gsf$ satisfies $\min_{\alpha \in \theta} \alpha(\kappa(g_n)) \to + \infty$ and $U_{\theta}(g_n) \rightarrow x$,  then 
$$
g_n o \rightarrow \iota(x) \quad \text{in } \overline{X}^\theta.
$$
\item For $\alpha \in \theta$, 
if $x \in \Fc_\theta$, $\alpha \in \theta$, and $v_\alpha \in V_\alpha$ is a unit vector with $\zeta_\alpha (x) = [v_\alpha]$, then 
$$
N_{\alpha} \omega_\alpha \iota(x)(ho) = \log \norm{\Phi_\alpha(h^{-1})v_\alpha}. 
$$
\item If $\mu$ is a (coarse, resp.) $(\Gamma, \phi, \delta)$-Patterson--Sullivan measure on $\Fc_\theta$ in the sense of Definition~\ref{defn:Quints definition}, then $\iota_* \mu$ is a (coarse, resp.) $(\Gamma, \phi, \delta)$-Patterson--Sullivan measure on $\partial_\theta X$ in the sense of Definition~\ref{defn: PS on X theta}.
\end{enumerate}
\end{proposition} 

\begin{proof} By   \cite[Lemma 6.6]{Quint_PS}, if a sequence $\{g_n\} \subset \Gsf$ satisfies  $\min_{\alpha \in \theta} \alpha(\kappa(g_n)) \to + \infty$ and $U_{\theta}(g_n) \rightarrow x$, then
$$
\lim_{n \rightarrow  + \infty} \pi_\theta b_{g_no}(ho) = \lim_{n \rightarrow  + \infty} \pi_{\theta} \kappa (h^{-1} g_n) - \pi_{\theta} \kappa (g_n) = B_\theta^{IW}(h^{-1},x)=\iota(x)(ho)
$$
for all $h \in \Gsf$. This shows that Equation~\eqref{eqn:definition of iota} defines a continuous map $\iota : \Fc_\theta \rightarrow \partial_\theta X$ and also establishes the first ``moreover'' part. Notice that the second ``moreover'' part is a consequence of the first, Lemma~\ref{lem:rank one limits in Valpha}, and Lemma~\ref{lem:Titsexpress}. 
The third ``moreover'' part is an immediate consequence of the definition of~$\iota$. 

To finish the proof we need to show that $\iota$ is a topological embedding. 
Since $\iota$ is continuous and $\Fc_\theta$ is compact, it suffices to show that $\iota$ is injective. Suppose that $\iota(x) = \iota(y)$. Let $x=k_1 \Psf_{\theta}$ and $y=k_2 \Psf_{\theta}$ where $k_1,k_2 \in \Ksf$. 

Fix $\alpha \in \theta$ and fix a unit vector $u_\alpha$ in $V_\alpha^+$. Property~\ref{item:boundary maps of Tits repn} implies that $\zeta_\alpha(x) = [\Phi_\alpha(k_1)u_\alpha]$ and $\zeta_\alpha(y) = [\Phi_\alpha(k_2)u_\alpha]$. Thus by part (2) of the moreover part of this proposition, 
$$
\norm{\Phi_\alpha(h^{-1}) \Phi_\alpha(k_1) u_\alpha} = e^{N_{\alpha} \omega_\alpha \iota(x)(ho)} =e^{N_{\alpha} \omega_\alpha \iota(y)(ho)} = \norm{\Phi_\alpha(h^{-1}) \Phi_\alpha(k_2) u_\alpha} 
$$
for all $h \in \Gsf$. Fix $H \in \mfa^+$ with $\min_{\alpha \in \theta} \alpha(H) > 0$. Then the above equation with $h := k_2 e^{-H}$ 
implies that 
$$
\norm{\Phi_\alpha(e^H)  \Phi_\alpha(k_2^{-1}k_1) u_\alpha} =  \norm{\Phi_\alpha(e^H)u_\alpha}.
$$
By Properties \ref{item:SVs of Tits repn} and \ref{item:splitting for Tits repn}, the map $v \in V_\alpha \smallsetminus \{0\} \mapsto \frac{  \norm{\Phi_\alpha(e^H)v}}{\norm{v}}$ is maximized only on $V_\alpha^+ \smallsetminus \{0\}$.  So we must have $ \Phi_\alpha(k_2^{-1}k_1) u_\alpha=\pm u_\alpha$. Then 
$$
\zeta_\alpha(x) = [\Phi_\alpha(k_1)u_\alpha] =  [\Phi_\alpha(k_2)u_\alpha] = \zeta_\alpha(y).
$$
Since this holds for all $\alpha \in \theta$, we have $x=y$. 
\end{proof}


\section{Shadows and contracting conical limit sets}\label{sec:shadows and cc limit set} 
 

In this section, we define shadows on $\partial_{\theta} X$ and use them to introduce the contracting conical limit set of a discrete subgroup. 

Recall, from Lemma~\ref{lem:Titsexpress}, that if $\xi \in \partial_\theta X$ and $g \in \Gsf$, then  
$$
\omega_\alpha\xi(g^{-1}o) \leq \omega_\alpha \kappa(g)
$$
for all $\alpha \in \theta$. Given $g \in \Gsf$, we introduce shadows in $\partial_\theta X$ by considering the set of functionals $\xi$ that are close to maximizing the expression $\omega_\alpha\xi(g^{-1}o)$ for all $\alpha \in \theta$. 

More precisely, for $g \in \Gsf$ and $R > 0$, define the associated \emph{shadow} by
$$
\Oc_R^\theta(g):= g \cdot \{    \xi \in \partial_{\theta} X : \omega_\alpha \xi(g^{-1}o) > \omega_\alpha \kappa(g) - R \text{ for all } \alpha \in \theta  \}.
$$

 In what follows we use $\pi_\theta$ to denote both the projection $\fa \rightarrow \fa_\theta$ satisfying Equation~\eqref{eqn:defn of pi_theta to mfa_theta} and the map $\partial_{\Delta} X \to \partial_{\theta} X$ obtained by the postcomposition with this projection. Since $\omega_\alpha \xi = \omega_\alpha \pi_\theta \xi$ for all $\alpha \in \theta$ and $\xi \in \partial_\Delta X$, we have 
$$
\pi_\theta \Oc_R^\Delta(g) \subset \Oc_R^\theta(g). 
$$
We also use Proposition \ref{prop:embedding of partial flags} to view $\Fc_\theta$ as a subset of $\partial_\theta X$. Then Equation~\eqref{eqn:defn of BIW}  and Proposition~\ref{prop:embedding of partial flags} imply that $\pi_\theta|_{\Fc_\Delta}$ coincides with the natural projection $\Fc_\Delta \rightarrow \Fc_\theta$ given by $g \Psf_\Delta \rightarrow g \Psf_\theta$.

Given a subgroup $\Ga < \Gsf$, we define its \emph{conical limit set} in $\partial_{\theta} X$ by
\begin{equation} \label{eqn:conical limit set classic}
   \La_{\theta}^{\rm con}(\Ga) := \left\{ \xi \in \partial_{\theta} X : \exists \, R > 0, \text{ escaping } \{ \ga_n \} \subset \Ga \text{ s.t. } \xi \in \bigcap_{n \ge 1} \Oc_R^{\theta}(\ga_n) \right\},
\end{equation}
following the classical definition of conical limit sets in rank one settings.

When $\Gsf$ is of higher rank, the intersection $\bigcap_{n \ge 1}\Oc_R^{\theta}(\ga_n)$ may not be a singleton, even after intersecting with the partial flag manifold $\Fc_{\theta}$. Moreover, the conical limit set $\La_{\theta}^{\rm con}(\Ga)$, after intersecting with $\Fc_{\theta}$, may not be a subset of the limit set $\La_{\theta}(\Ga)$. See Example \ref{example:nonshrinking conical} below for detailed descriptions of them. Hence, in view of Lemma \ref{lem:diam decay} below, we define the following smaller subset of conical limit set, which only involves shrinking shadows.

\begin{definition} \label{def:concon}
   
Given a subgroup $\Ga < \Gsf$, we call $\xi \in \partial_{\theta} X$ a \emph{contracting conical limit point} of $\Ga$ if there exist $R > 0$ and a sequence $\{ \ga_n \} \subset \Ga$ such that
$$
\lim_{n \to + \infty} \min_{\alpha \in \theta} \alpha(\kappa(\ga_n)) = + \infty \quad \text{and} \quad \xi \in \bigcap_{n \ge 1} \Oc_{R}^{\theta}(\ga_n).
$$
We denote by $\La_{\theta}^{\rm concon}(\Ga)$ the \emph{contracting conical limit set} of $\Ga$, which is defined as the set of all contracting conical limit points of $\Ga$. 
\end{definition}

For general groups, these limit sets have the following properties.

\begin{proposition}\label{prop:concon limit set for general groups} If $\Ga < \Gsf$ is a subgroup, then both $\La_{\theta}^{\rm concon}(\Ga)$ and $\La_{\theta}^{\rm con}(\Ga)$ are $\Ga$-invariant subsets. Moreover, $\La_{\theta}^{\rm concon}(\Ga)$ is a subset of the limit set $\Lambda_\theta(\Gamma) \subset~\Fc_\theta$ introduced in Section~\ref{sec:transverse groups}. \end{proposition} 

For transverse groups, one can say more. 

\begin{proposition}\label{prop:concon limit set for transverse groups} If $\Ga < \Gsf$ is a non-elementary $\Psf_{\theta}$-transverse group, then:
\begin{enumerate} 
\item $\La_{\theta}^{\rm concon}(\Ga)=\La_\theta^{\rm con}(\Gamma)$. Hence $\La_\theta^{\rm con}(\Gamma)$ is a subset of the limit set $\Lambda_\theta(\Gamma) \subset \Fc_\theta$ introduced in Section~\ref{sec:transverse groups}. 
\item  $\La_\theta^{\rm con}(\Gamma)$ coincides  with the conical limit set in the convergence group sense (recall that $\Gamma$ acts on $\Lambda_\theta(\Gamma)$ as a convergence group). 
\end{enumerate}
\end{proposition} 

In the next two subsections we establish some properties of these shadows and relate them to symmetric space shadows. Then in Section~\ref{sec:proof of shadow propositions} we prove Propositions~\ref{prop:concon limit set for general groups} and~\ref{prop:concon limit set for transverse groups}.

\subsection{Properties of shadows} We record some properties of shadows.

\begin{lemma} \label{lem:shadow translate shadow}
   For any  $g \in \Gsf$ and $R > 0$, there exists $R' = R'(g, R) > 0$ such that: if $h \in \Gsf$, then 
   $$
   g\Oc_R^{\theta}(h) \subset \Oc_{R'}^{\theta}(gh).
   $$
   In particular, $\La_{\theta}^{\rm con}(\Ga)$ and $\La_{\theta}^{\rm concon}(\Ga)$ are $\Ga$-invariant.

\end{lemma}

\begin{proof}
   Fix $\xi \in \partial_{\theta} X$ with $ h \xi \in  \Oc_R^{\theta}(h)$. Then  for each $\alpha \in \theta$,
   $$
   \omega_{\alpha} \xi(h^{-1} o) > \omega_{\alpha} \kappa(h) - R.
   $$
   Then by Lemma \ref{lem:Lipschitz}  and Property \ref{item:SVs of Tits repn}, there exists $C = C(g) > 0$ such that
   $$
   \abs{ \omega_{\alpha} \xi(h^{-1} o) - \omega_{\alpha} \xi(h^{-1}g^{-1} o)} < C \quad \text{and} \quad \abs{\omega_{\alpha} \kappa(h) - \omega_{\alpha} \kappa(gh) } < C.
   $$
   Hence, we have
   $$
   \omega_{\alpha} \xi(h^{-1} g^{-1} o) > \omega_{\alpha} \kappa(gh) - R - 2C.
   $$
   Setting $R' = R + 2C$, this implies $gh  \xi \in \Oc_{R'}^{\theta}(gh)$, as desired.
\end{proof}

The next lemma shows that a shadow contains the ``endpoint'' of the associated group element,  hence motivating the terminology.

\begin{lemma} \label{lem:endpoint in shadow}
   For $g \in \Gsf$ with $\min_{\alpha \in \theta} \alpha(\kappa(g)) > 0$, we have
   $$
   U_{\theta}(g) \in \Oc_R^{\theta}(g)
   $$
   for all $R > 0$. 
\end{lemma}

\begin{proof} Fix a Cartan decomposition $g = ka \ell \in \Ksf \Asf^+ \Ksf$. Then $U_\theta(g) = k \Psf_\theta$. 

   Fix any sequence $\{H_n \} \subset  \mfa^+$ such that $\alpha(H_n) \to + \infty$ for all $\alpha \in \Delta$. Then viewing $k \Psf_\Delta$ as an element of $\partial_\Delta X$, Proposition~\ref{prop:embedding of partial flags} implies that 
$$
(k \Psf_{\Delta})(h o) = \lim_{n \rightarrow  + \infty} b_{k e^{H_n}o}(ho) = \lim_{n \rightarrow  + \infty} \kappa( h^{-1} k e^{H_n}) - \kappa(ke^{H_n})
$$
for all $h \in \Gsf$.
If $a = e^H$ where $H \in \mfa^+$, then 
\begin{align*}
g^{-1}\cdot (k \Psf_{\Delta} )(g^{-1} o) & = (k \Psf_{\Delta})(g g^{-1} o) - (k \Psf_{\Delta})(g o)\\
& =0 - \lim_{n \rightarrow  + \infty} \kappa( \ell^{-1}e^{-H}k^{-1}ke^{H_n}) - \kappa(ke^{H_n}) \\
& = -\lim_{n \rightarrow  + \infty} H_n-H - H_n= H = \kappa(a) = \kappa(g). 
\end{align*}
So $k \Psf_\Delta \in  \Oc_R^{\Delta}(g)$, which implies that 
\begin{equation*} 
 U_{\theta}(g) = \pi_\theta(k \Psf_\Delta )  \in \Oc_R^{\theta}(g). \qedhere
 \end{equation*}
\end{proof}

For the next result, we fix any metric generating the topology on $\overline{X}^\theta$.

\begin{lemma} \label{lem:diam decay}
For a sequence $\{g_n\} \subset \Gsf$, if $\min_{\alpha \in \theta} \alpha(\kappa(g_n)) \rightarrow +\infty$, then 
$$
 {\rm diam} \Oc_R^\theta(g_n) \rightarrow 0.
$$
\end{lemma}

\begin{proof}
Fix $R > 0$ and a sequence $\{g_n\} \subset \Gsf$ such that $\alpha(\kappa(g_n)) \to + \infty$ for all $\alpha \in \theta$. It suffices to consider the case where $g_n o \rightarrow \xi \in \partial_\theta X$. 

Suppose the lemma fails. Then, after possibly passing to a subsequence of $\{g_n\}$, we can find a sequence $\{\xi_n\} \subset \partial_{\theta} X$ such that 
$$g_n \xi_n \in \Oc_{R}^{\theta}(g_n) \text{ for all } n \in \Nb \quad \text{ and } \quad g_n \xi_n \rightarrow \eta \neq \xi.$$
 
 After passing to a subsequence, we can suppose that $\frac{\Phi_{\alpha}(g_n)}{\norm{\Phi_{\alpha}(g_n)}} \to S_{\alpha}$ in  $\End(V_\alpha)$ for all $\alpha \in \theta$. Then by Lemma \ref{lem:Titsexpress}, 
 $$
N_{\alpha} \omega_{\alpha}  \xi(h o)   = \log \norm{ \Phi_{\alpha}(h^{-1}) S_\alpha  }
$$
 for all $h \in \Gsf$ and $\alpha \in \theta$. Further, Lemma~\ref{lem:rank one limits in Valpha} implies that each $S_\alpha$ has rank one. 

 For each $n$ and $\alpha \in \theta$, using Lemma \ref{lem:Titsexpress} we can fix $T_{\alpha}^{\xi_n} \in \End(V_{\alpha})$ with $\norm{T_{\alpha}^{\xi_n}} = 1$ such that 
$$\begin{aligned}
N_{\alpha} \omega_{\alpha} g_n  \xi_n(h o) & = N_{\alpha} \omega_{\alpha} \xi_n (g_n^{-1} h o) - N_{\alpha} \omega_{\alpha} \xi_n (g_n^{-1} o) \\
&  = \log \norm{ \Phi_{\alpha}(h^{-1}) \frac{\Phi_{\alpha}(g_n) T_{\alpha}^{\xi_n}}{\norm{\Phi_{\alpha}(g_n) T_{\alpha}^{\xi_n}}} }
\end{aligned}
$$
for all $h \in \Gsf$. Passing to a subsequence, we can suppose that $T_\alpha^{\xi_n} \rightarrow T_\alpha$ in $\End(V_{\alpha})$ for all $\alpha \in \theta$. 

Since $g_n  \xi_n \in \Oc_R^{\theta}(g_n)$, we have that
$$
\omega_{\alpha} \xi_n (g_n^{-1} o) \ge \omega_{\alpha} \kappa(g_n) - R = \frac{1}{N_{\alpha}}\log \norm{\Phi_{\alpha}(g_n)}-R.
$$
So 
$$
\norm{ \frac{\Phi_{\alpha}(g_n) }{\norm{\Phi_{\alpha}(g_n)}} T_{\alpha}^{\xi_n}} \ge e^{-N_{\alpha} R},
$$
which implies that $S_\alpha T_\alpha \neq 0$. Then, since $g_n  \xi_n \rightarrow \eta$, we then have 
$$
N_{\alpha} \omega_{\alpha}  \eta(h o)  = \log \norm{ \Phi_{\alpha}(h^{-1}) \frac{S_\alpha T_{\alpha}}{\norm{S_\alpha T_{\alpha}}} }
$$
for all $\alpha \in \theta$ and $h \in \Gsf$. Notice that $\frac{S_\alpha T_{\alpha}}{\norm{S_\alpha T_{\alpha}}}$ has rank one, operator norm one, and the same image as $S_\alpha$. Since $\norm{S_{\alpha}} = 1$ as well, we have $\frac{S_\alpha T_\alpha}{\norm{S_\alpha T_\alpha}} = S_\alpha U$ for some orthogonal matrix $U \in \SL(V_\alpha)$ (recall that $V_\alpha$ has a fixed inner product). 
Thus 
$$
N_{\alpha} \omega_{\alpha}  \eta(h o)  = \log \norm{ \Phi_{\alpha}(h^{-1}) S_\alpha } = N_{\alpha} \omega_{\alpha}  \xi(h o)  
$$
for all $\alpha \in \theta$ and $h \in \Gsf$. Thus $\eta = \xi$ and we have a contradiction. 
\end{proof}

\subsection{Shadows from symmetric spaces} \label{section:symm space shadows}

On $\Fc_{\theta}$, another natural definition of shadows involves balls and flats in the symmetric space. As before, let $X = \Gsf / \Ksf$ denote the symmetric space associated to $\Gsf$ endowed with a symmetric metric and let $o = K \in X$.
 
 For $R > 0$ and $g \in \Gsf$, the \emph{symmetric space shadow} $O_R^{\theta}(o, g o) \subset \Fc_{\theta}$ of the ball $B_X(g o, R) \subset X$ of radius $R > 0$ and center $g o$  is defined by
$$
O_R^{\theta}(o, go) := \left\{ k \Psf_{\theta} \in \Fc_{\theta} : k \in \Ksf \text{ and }  k \Asf^+ o \cap B_X(g o, R) \neq \emptyset\right\}.
$$

Notice that $O_R^{\theta}(o, g o)$ is an open subset of $\Fc_{\theta}$ while $\Oc_R^{\theta}(g)$ is an open subset of a larger space, namely $\partial_{\theta} X$. The following proposition relates the two shadows and part (2) implies that the  contracting conical limit set could be defined using symmetric space shadows.

\begin{proposition} \label{prop:comparision with symm shadows}
   \
   \begin{enumerate}
      \item For any $R > 0$, there exists $r > 0$ such that 
      $$
      O_R^{\theta}(o, g o) \subset \Oc_{r}^{\theta}(g) \quad \text{for all} \quad g \in \Gsf.
      $$
      \item For any $R > 0$, there exists $r > 0$ such that if  $\{ g_n \} \subset \Gsf$ is a sequence with $\min_{\alpha \in \theta} \alpha(\kappa(g_n)) \to + \infty$ and $ \bigcap_{n \geq 0} \Oc_R^\theta(g_n) \neq \emptyset$, then 
       $$
      \bigcap_{n \geq 0} \Oc_R^\theta(g_n) =  \bigcap_{n \geq N_0} O_{r}^{\theta}(o, g_n o) \quad \text{ for some $N_0$. }
      $$
   \end{enumerate}
\end{proposition}

\begin{proof}
   We first prove the part (1). Fix $R > 0$. Then there exists $C > 0$ such that
   $$
   \norm{ B_{\theta}^{IW}(g, g^{-1} x) - \pi_\theta\kappa(g)} \le C
   $$
   for all $g \in \Gsf$ and $x \in O_R^{\theta}(o, g o)$ (\cite[Lemma 5.7]{LO_Invariant}, \cite[Lemma 5.10]{KOW_PD}). Hence, there exists $r > 0$ such that for each $\alpha \in \theta$,
   $$
   \omega_{\alpha} B_{\theta}^{IW}(g, g^{-1} x) > \omega_{\alpha} \kappa(g) - r
   $$
   for all $g \in \Gsf$ and $x \in O_R^{\theta}(o, g o)$. By Proposition \ref{prop:embedding of partial flags}, this implies
   $$O_R^{\theta}(o, go) \subset \Oc_r^{\theta}(g).$$

   We now prove the part (2). We start by proving the following weaker claim: 
   
   \medskip

\noindent \textbf{Claim:} For any $R > 0$, there exists $r > 0$ such that if  $\{ g_n \} \subset \Gsf$ is a sequence with $\min_{\alpha \in \theta} \alpha(\kappa(g_n)) \to + \infty$ and $ \bigcap_{n \geq 0} \Oc_R^\theta(g_n) \neq \emptyset$, then 
       $$
      \bigcap_{n \geq 0} \Oc_R^\theta(g_n) \subset  \bigcap_{n \geq N_0} O_{r}^{\theta}(o, g_n o) \quad \text{ for some $N_0$. }
      $$

      \medskip

\noindent \emph{Proof of Claim:}  Suppose not. Then for each $k \in \Nb$, there exist a sequence $\{g_{k, n} \} \subset \Gsf$ with $\min_{\alpha \in \theta} \alpha(\kappa (g_{k, n})) \to + \infty$ as $n \to + \infty$ and $\{x_k \} = \bigcap_{n \in \Nb} \Oc_R^{\theta}(g_{k, n})$ while $x_k \notin O_{k}^{\theta}(o, g_{k, n} o)$ for infinitely many $n \in \Nb$. Lemmas \ref{lem:endpoint in shadow} and \ref{lem:diam decay} imply that $x_k \in \Fc_{\theta}$.

   Then for each $k \in \Nb$, we can choose $n_k \in \Nb$ so that $x_k \notin O_k^{\theta}(o, g_{k, n_k} o)$ and $\min_{\alpha \in \theta} \alpha(\kappa(g_{k, n_k})) \ge k$ for all $k \in \Nb$. Setting $g_k := g_{k, n_k}$ for simplicity, we have
   $$
   g_k^{-1} x_k \notin g_k^{-1} O_k^{\theta}( o, g_k o) \quad \text{for all } k \in \Nb.
   $$
   After passing to a subsequence, we may assume that $g_k^{-1} x_k \to z \in \Fc_{\theta}$ and $U_{\opp^* \theta}(g_k^{-1}) \to y \in \Fc_{\opp^* \theta}$. By \cite[Proposition 3.4]{KOW_SF}, we have that $y$ and $z$ are not transverse.

   On the other hand, since $x_k \in \Oc_R^{\theta}(g_k)$, Proposition \ref{prop:embedding of partial flags} implies that for each $\alpha \in \theta$,
   $$
  \omega_{\alpha} B_{\theta}^{IW} (g_k, g_k^{-1} x_k) =\omega_{\alpha} \left( (g_k^{-1} \cdot x_k)(g_k^{-1} o)\right) > \omega_{\alpha} \kappa(g_k) - R.
   $$
For $\alpha \in \theta$ and $k \in \Nb$, fix a unit vector $v_{\alpha, k} \in V_{\alpha}$ with $\zeta_{\alpha}(g_k^{-1} x_k) = [v_{\alpha, k}]$. By Proposition~\ref{prop:embedding of partial flags},
   $$
   \omega_{\alpha} B_{\theta}^{IW} (g_k, g_k^{-1} x_k) =\frac{1}{N_{\alpha}} \log  \norm{ \Phi_{\alpha}(g_k) v_{\alpha, k}}.
   $$
Hence by Property~\ref{item:SVs of Tits repn},
   $$
   \norm{ \frac{\Phi_{\alpha}(g_k)}{\norm{\Phi_{\alpha}(g_k)}} v_{\alpha, k}} > e^{- N_{\alpha} R}.
   $$
   Taking the limit $k \to + \infty$, we may assume that $v_{\alpha, k} \to v_\alpha \in V_{\alpha}$ and $\frac{\Phi_{\alpha}(g_k)}{\norm{\Phi_{\alpha}(g_k)}} \to T_{\alpha} \in \End(V_{\alpha})$. Then we have
   $$
   \norm{T_{\alpha} v_\alpha} \ge e^{-N_{\alpha} R}.
   $$
   In particular, $\zeta_{\alpha} (z) = [v_\alpha] \notin \ker T_{\alpha} = \zeta_{\alpha}^*  (y)$ by Lemma \ref{lem:rank one limits in Valpha}. Since this holds for all $\alpha \in \theta$, $y$ and $z$ are transverse by Property \ref{item:boundary maps of Tits repn}(b). This is a contradiction, finishing the proof of the claim. \hfill $\blacktriangleleft$ 
   
   \medskip
   
Now suppose $\{g_n\}$ satisfies the hypothesis of part (2). By part (1) we can fix $R'> R$ such that $O_r^\theta(o,go) \subset \Oc_{R'}^\theta(g)$ for all $g \in \Gsf$. Then  
 $$
      \bigcap_{n \geq 0} \Oc_R^\theta(g_n) \subset  \bigcap_{n \geq N_0} O_{r}^{\theta}(o, g_n o) \subset  \bigcap_{n \geq 0} \Oc_{R'}^\theta(g_n) =  \bigcap_{n \geq 0} \Oc_R^\theta(g_n), 
      $$
   where the last equality uses Lemma \ref{lem:diam decay}. 
\end{proof}

\begin{example} \label{example:nonshrinking conical}
Consider  $\Gsf = \PSL(2, \Rb) \times \PSL(2, \Rb)$. Then
\begin{itemize}
   \item the associated symmetric space is $\Hb^2 \times \Hb^2$ with the symmetric space distance $\dist ((x_1, x_2), (y_1, y_2)) = \sqrt{ \dist_{\Hb^2}(x_1, y_1)^2 + \dist_{\Hb^2}(x_2, y_2)^2 }$,
   \item $\Delta = \{ \alpha_1, \alpha_2 \}$, where $\alpha_1$ and $\alpha_2$ are simple roots for the first and the second $\PSL(2, \Rb)$-components respectively, and
   \item the Furstenberg boundary is $\partial \Hb^2 \times \partial \Hb^2$.
\end{itemize} 

Setting
$$
\Ga := \Ga_0 \times \{ \id \} < \Gsf
$$
where $\Ga_0 < \PSL(2, \Rb)$ is a cocompact lattice, we compute its limit set and conical limit set.

First, since all elements of $\Ga$ have identity on their second component, we have
$$
\La_{\Delta}(\Ga) = \emptyset.
$$

Fix a basepoint $o = (o_1, o_2) \in \Hb^2 \times \Hb^2$. Then for $R > 0$ and $\ga = (\ga_0, \id) \in \Ga$, the symmetric space shadow is
$$
O_R (o, \ga o) = \left\{ (x, y) \in \partial \Hb^2 \times \partial \Hb^2 : \begin{matrix}
\exists \text{ geodesic ray from } o_1 \text{ to } x \text{ in } \Hb^2 \\
\text{intersecting } B_{\Hb^2}(\ga_0 o_1, R)
\end{matrix} \right\}.
$$
In particular, $O_R(o, \ga o) = O_R(o_1, \ga_0 o_1) \times \partial \Hb^2$. Hence for any escaping $\{ \ga_n\} \subset \Ga$ and $R > 0$, $\bigcap_{n \ge 1} O_R (o, \ga_n o)$ is not a singleton as long as it is non-empty. In addition, by Proposition \ref{prop:comparision with symm shadows} and the fact that $\Ga_0$ is a cocompact lattice, the above observation implies that
$$
\La_{\Delta}^{\rm con}(\Ga) = \partial \Hb^2 \times \partial \Hb^2.
$$

\end{example}

\subsection{Proofs of Propositions~\ref{prop:concon limit set for general groups} and~\ref{prop:concon limit set for transverse groups}\label{sec:proof of shadow propositions}}

\begin{proof}[Proof of Proposition~\ref{prop:concon limit set for general groups}] Lemma \ref{lem:shadow translate shadow} implies that $\Lambda^{\rm concon}_\theta(\Gamma)$ and $\La_{\theta}^{\rm con}(\Ga)$ are $\Gamma$-invariant. Lemmas \ref{lem:endpoint in shadow} and \ref{lem:diam decay} imply that $\Lambda^{\rm concon}_\theta(\Gamma) \subset \Lambda_\theta(\Gamma)$. 
\end{proof} 

\begin{proof}[Proof of Proposition~\ref{prop:concon limit set for transverse groups}] The first assertion in part (1) follows from the definition of a transverse group and the ``hence'' part follows from Proposition~\ref{prop:concon limit set for general groups} and the first assertion.

For part (2), first suppose that $x \in \La_{\theta}^{\rm con}(\Ga)= \La_{\theta}^{\rm concon}(\Ga)$. Then there exist $R > 0$ and a sequence $\{ \ga_n \} \subset \Ga$ such that $\alpha(\kappa(\ga_n)) \to + \infty$ for all $\alpha \in \theta$ and $x \in \bigcap_{n \ge 1} \Oc_R^{\theta}(\ga_n)$. Then Lemmas \ref{lem:endpoint in shadow} and \ref{lem:diam decay} imply that $U_{\theta}(\ga_n) \to x$. Then it follows from Proposition \ref{prop:comparision with symm shadows}(2) that $x \in \bigcap_{n \ge 1} O_r^{\theta}(o, \ga_n o)$ for some $r > 0$. By \cite[Lemma 5.8]{KOW_PD}, for any $y \in \Fc_{\opp^* \theta}$ transverse to $x$, the sequence $\ga_n^{-1}(x, y)$ converges to a transverse pair in $\Fc_{\theta} \times \Fc_{\opp^* \theta}$. Since any two distinct points in $\La_{\theta \cup \opp^* \theta}(\Gamma)$ are transverse and the projections $\La_{\theta \cup \opp^* \theta}(\Gamma) \to \La_{\theta}(\Gamma)$ and $\La_{\theta \cup \opp^* \theta}(\Gamma) \to \La_{\opp^* \theta}(\Gamma)$ are $\Ga$-equivariant homeomorphisms (Observation \ref{obs:symmetry theta}), this implies that $x$ is a conical limit point in the sense of the convergence action of $\Ga$ on $\La_{\theta}(\Ga)$.

Conversely, suppose that $x \in \La_{\theta}(\Ga)$ is a conical limit point in the sense of the convergence action of $\Ga$ on $\La_{\theta}(\Ga)$. Then there exist $a,b \in \Lambda_\theta(\Gamma)$ distinct and an escaping sequence $\{\gamma_n\} \subset \Gamma$ such that $\gamma_n^{-1}x \rightarrow a$ and $\gamma_n^{-1}y \rightarrow b$ for all $y \in \Lambda_\theta(\Gamma) \smallsetminus \{x\}$. Then Proposition \ref{prop:NS dynamics}  implies that $U_\theta(\gamma_n) \rightarrow x$. Further, since the projections  $\La_{\theta \cup \opp^* \theta}(\Gamma) \to \La_{\theta}(\Gamma)$ and $\La_{\theta \cup \opp^* \theta}(\Gamma) \to \La_{\opp^* \theta}(\Gamma)$ are $\Ga$-equivariant homeomorphisms,  $\ga_n^{-1}(x, y)$ converges to a transverse pair in  $\Fc_{\theta} \times \Fc_{\opp^* \theta}$ for any $y \in \La_{\opp^* \theta}(\Ga)$ which is transverse to $x$. By \cite[Lemma 5.8]{KOW_PD}, this implies that $x \in \bigcap_{n \ge 1} O_R^{\theta}(o, \ga_n o)$ for some $R > 0$. Then by Proposition \ref{prop:comparision with symm shadows}(1), $x \in \La_{\theta}^{\rm concon}(\Ga)$.
\end{proof}


\section{Verifying the PS-system axioms}\label{sec:verifying the PS axioms} 


In this section, we consider a discrete subgroup $\Ga < \Gsf$ and verify the axioms of PS-systems for the boundaries $\overline{X}^\theta$.

\begin{theorem} \label{thm:Zdense PS}

Suppose $\theta \subset \Delta$ and $\phi \in \mfa_\theta^*$. 
   If $\Gamma < \Gsf$ is strongly $(\Phi_{\alpha})_{\alpha \in \theta}$-irreducible and $\mu$ is a coarse $(\Gamma, \phi, \delta)$-Patterson--Sullivan measure on $\partial_{\theta} X$, then $(\partial_{\theta} X, \Gamma, \phi \circ B_{\theta}, \mu)$ is a PS-system, with  magnitude $\norm{\ga}_{\phi} := \phi(\kappa(\ga))$ and the  $R$-shadows $\Oc_R^{\theta}(\ga)$ for each $\ga \in \Ga$. Moreover, \ref{item:baire} holds.

\end{theorem}   

For transverse groups, we can show that the system is well-behaved. 

\begin{theorem} \label{thm:transverse well-behaved}
Suppose $\theta \subset \Delta$, $\phi \in \mfa_\theta^*$, and $\Gamma < \Gsf$ is a strongly $(\Phi_{\alpha})_{\alpha \in \theta}$-irreducible  $\Psf_{\theta}$-transverse group. Let $\mu$ be a coarse $(\Gamma, \phi, \delta)$-Patterson--Sullivan measure on $\partial_{\theta} X$. Then the PS-system $(\partial_{\theta}X , \Gamma, \phi \circ B_{\theta}, \mu)$ in Theorem \ref{thm:Zdense PS} is well-behaved with respect to the trivial hierarchy $\mathscr{H}(R) \equiv \Ga$.

\end{theorem}

Using work in~\cite{CZZ2024} we will show that for transverse groups, divergence of the Poincar\'e series implies that there is a unique PS-measure.

\begin{theorem}\label{thm:uniqueness for transverse}  Suppose $\theta \subset \Delta$ and $\Gamma < \Gsf$ is a strongly $(\Phi_{\alpha})_{\alpha \in \theta}$-irreducible  $\Psf_{\theta}$-transverse group. If  $\phi \in \mfa_\theta^*$, $\delta^{\phi}(\Ga) < + \infty$, and  
$$
\sum_{\ga \in \Ga} e^{-\delta^{\phi}(\Ga) \phi(\kappa(\ga))} = + \infty,
$$
then there is a  unique $(\Ga, \phi, \delta^{\phi}(\Ga))$-Patterson--Sullivan measure $\mu$ on $\partial_\theta X$, the $\Gamma$-action on $(\partial_\theta X, \mu)$ is ergodic, and 
   $$
   \mu(\La_{\theta}^{\rm con}(\Ga)) = 1
   $$
(in particular, $\mu$ is supported on $\Fc_\theta$).
\end{theorem} 

\begin{remark} Previously, Canary, Zhang, and the second author proved in \cite{CZZ2024}  uniqueness for measures supported on the limit set $\Lambda_\theta(\Gamma) \subset \Fc_\theta$. When $\Ga < \Gsf$ is Zariski dense, the first author, Oh, and Wang  previously proved uniqueness for measures supported on $\Fc_{\theta}$ in \cite{KOW_PD}. Since a  Zariski dense subgroup is   strongly $(\Phi_{\alpha})_{\alpha \in \theta}$-irreducible, the above theorem generalizes this uniqueness result. \end{remark}

In the arguments that follow it will be helpful to have the following terminology. Given $\xi \in \partial_\theta X$, we say that $(T_\alpha^\xi)_{\alpha \in \theta} \in \prod_{\alpha \in \theta} \End(V_\alpha)$  \emph{represents} $\xi$ if 
\begin{itemize}
\item $N_{\alpha} \omega_{\alpha} \xi(go)  = \log \norm{\Phi_{\alpha}(g^{-1})T_\alpha^\xi}$ for all $g \in \Gsf$ and all $\alpha \in \theta$, 
\item each $T_\alpha^\xi$ is a limit of elements of the form $\frac{1}{\norm{\Phi_\alpha(g)}}\Phi_\alpha(g)$.
\end{itemize}
Lemma \ref{lem:Titsexpress} implies that every element $\xi \in \partial_\theta X$ is represented by such a list, but the representation is not unique. For instance, one can always right-multiply by  elements in $\Phi_\alpha(\Ksf)$. 

\subsection{Proof of Theorem~\ref{thm:Zdense PS}} We verify Property \ref{item:coycles are bounded}, Property \ref{item:almost constant on shadows}, and Property \ref{item:baire}. 
Since Property \ref{item:baire} implies Property \ref{item:empty Z intersection}, this completes the proof of the theorem. 

Property \ref{item:coycles are bounded}: Recall that for $g \in \Gsf$ and $\xi \in \partial_{\theta} X$, $B_\theta(g, \xi) = \xi(g^{-1}o)$. Since $\xi(o) =0$ for all $\xi \in \partial_\theta X$,  Lemma \ref{lem:Lipschitz} implies Property \ref{item:coycles are bounded}. 

Property \ref{item:almost constant on shadows}: By Lemma \ref{lem:Titsexpress}, for any $\xi \in \partial_{\theta} X$, $g \in \Gsf$, and $\alpha \in \theta$ we have
$$
\omega_{\alpha} B_{\theta}(g, \xi)=\omega_{\alpha} \xi(g^{-1}o)  \le \omega_{\alpha} \kappa(g).
$$
Further, for $\xi \in g^{-1} \Oc_R^{\theta}(g)$, we have 
$$
\omega_{\alpha} \kappa(g) - R \le \omega_\alpha \xi(g^{-1}o) = \omega_{\alpha} B_{\theta}(g, \xi).
$$
Since $\phi \in \fa_{\theta}^*$ is a linear combination of the $\{ \omega_{\alpha} : \alpha \in \theta \}$, this implies  Property \ref{item:almost constant on shadows}.

Property \ref{item:baire}:  Let $\{\ga_n\} \subset \Ga$ and $R_n > 0$ be sequences such that $R_n \to + \infty$ and 
$$
\left[ \partial_{\theta} X \smallsetminus \ga_{n}^{-1} \Oc_{R_n}^{\theta}(\ga_{n}) \right] \to Z
$$
with respect to the Hausdorff distance. 

Then 
$$\begin{aligned}
\partial_{\theta} X \smallsetminus \ga_{n}^{-1} \Oc_{R_n}^{\theta}(\ga_{n}) & = \bigcup_{\alpha \in \theta} \left\{   \xi \in \partial_{\theta} X : \omega_{\alpha} \xi(\ga_{n}^{-1} o) \leq \omega_{\alpha} \kappa(\ga_{n}) - R_n \right\} \\
& = \bigcup_{\alpha \in \theta} \left\{\xi \in \partial_{\theta} X : \norm{ \frac{\Phi_{\alpha}(\ga_{n}) T_{\alpha}^{\xi}}{\norm{\Phi_{\alpha}(\ga_{n})}}} \le e^{-N_{\alpha} R_n}  \begin{array}{cc} \text{for all } T_\alpha^\xi \\ \text{representing } \xi\end{array}\right\}.
\end{aligned}
$$
After passing to a subsequence, we can suppose that $\frac{\Phi_{\alpha}(\ga_{n})}{\norm{\Phi_{\alpha}(\ga_{n})}} \to S_{\alpha} \in \End(V_{\alpha})$ for each $\alpha \in \theta$. Then 
$$
Z \subset \bigcup_{\alpha \in \theta} \left\{   \xi \in \partial_{\theta} X : S_{\alpha}T_{\alpha}^{\xi} = 0 \begin{array}{cc} \text{for all } T_\alpha^\xi \\ \text{representing } \xi\end{array}\right\}.
$$

Now fix $h_1,\dots, h_m \in \Gamma$ and $\xi \in Z$. Fix a representative $(T_\alpha^\xi)_{\alpha \in \theta}$ of $\xi$. Using Lemma \ref{lem:rank one limits in Valpha} and Lemma \ref{lem:rho irr implies flag irr}, we can fix $\ga \in \Ga$ such that 
$$
 S_{\alpha}\Phi_\alpha(h_j^{-1} \gamma)T_\alpha^\xi \neq 0
 $$
for all $\alpha \in \theta$ and all $1 \leq j \leq m$. We claim that 
$$
\gamma \xi \notin \bigcup_{j=1}^m h_j Z. 
$$
Notice that 
$$
(h_j^{-1}\gamma \xi)(go)  =\xi( \gamma^{-1} h_j g o) -\xi(\gamma^{-1} h_j o) 
$$
and so 
$$
\left( T_\alpha^{h_j^{-1} \gamma \xi} \right)_{\alpha \in \theta}:= \left( \frac{1}{\norm{\Phi_\alpha(h_j^{-1} \gamma )T_\alpha^{\xi}}} \Phi_\alpha(h_j^{-1} \gamma )T_\alpha^\xi\right)_{\alpha \in \theta}
$$
is a  representative of $h_j^{-1} \gamma \xi$. Further,  
$$
S_\alpha T_\alpha^{h_j^{-1} \gamma \xi} \neq 0
$$
for all $\alpha \in \theta$ and all $1 \leq j \leq m$. So 
$$
h_j^{-1} \gamma \xi \notin Z
$$
for all $1 \leq j \leq m$. Hence 
$$
\gamma \xi \notin \bigcup_{j=1}^m h_j Z. 
$$
Thus Property \ref{item:baire} holds. 
\qed

\subsection{Proof of Theorem~\ref{thm:transverse well-behaved}}From Theorem~\ref{thm:Zdense PS} we know that Properties \ref{item:coycles are bounded},  \ref{item:almost constant on shadows}, \ref{item:empty Z intersection}, and \ref{item:baire} hold.  Property \ref{item:shadow inclusion} follows from the definition of shadows. Property \ref{item:diam goes to zero ae} follows from Lemma \ref{lem:diam decay} and the definition of a transverse group. It remains to verify Properties \ref{item:properness} and \ref{item:intersecting shadows}. 

We first verify Property \ref{item:properness}.

\begin{lemma} \label{lem:properness}
   For any $T > 0$, the set $\{ \gamma \in \Gamma : \phi(\kappa(\gamma)) \leq T\}$ is finite. 

\end{lemma} 

\begin{proof} For Patterson--Sullivan measures supported on $\Lambda_\theta(\Gamma)$ this was verified in \cite[Proposition 10.1]{BCZZ_counting} and we will reduce to this case.

Suppose for a contradiction that there exists $T > 0$ and an infinite sequence of distinct elements $\{\gamma_n\} \subset \Gamma$ with $\phi(\kappa(\gamma_n)) \leq T$ for all $n$. Passing to a subsequence we can suppose that $U_\theta(\gamma_n) \rightarrow x$. Then $x \in \La_{\theta}(\Ga)$. Fix $R > 0$ large enough to satisfy the Shadow Lemma (Proposition~\ref{prop.shadowlemma}). Then there exists $\epsilon > 0$ such that 
$$
\mu( \Oc_R^{\theta}(\gamma_n)) \geq \epsilon
$$
for all $n \geq 1$. Further, for any open set $\Uc$ containing $x$ in $\partial_\theta X$, Lemma \ref{lem:endpoint in shadow} and  Lemma~\ref{lem:diam decay} imply that 
$$
\Oc_R^{\theta}(\gamma_n) \subset \Uc
$$
for all $n$ sufficiently large. So $\mu(\Uc) \geq\epsilon$. Since $\Uc$ is an arbitrary open set containing $x$, we must have $\mu(\{x\}) \geq \epsilon$. Then 
$$
\mu'(\cdot): = \frac{1}{\mu(\Gamma x)} \mu( \Gamma x \cap (\cdot))
$$
defines a coarse $(\Gamma, \phi, \delta)$-Patterson--Sullivan measure supported on $\Lambda_\theta(\Gamma)$. Now by \cite[Proposition 10.1]{BCZZ_counting}, the set $\{ \gamma \in \Gamma : \phi(\kappa(\gamma)) \leq T\}$ is finite, which is a contradiction. 
\end{proof} 

We establish two lemmas before proving Property \ref{item:intersecting shadows}. The following is a special case of \cite[Proposition 3.3(7)]{BCZZ_coarse}.

\begin{lemma} \label{lem:BCZZ converging to the same point}
For any $\{g_n \}, \{ h_n \} \subset \Ga$ such that $\phi(\kappa(g_n)) \le \phi(\kappa(h_n))$ for all $n \ge 1$ and $\{ g_n^{-1} h_n \}$ is escaping, we have
   $$
   \dist( U_{\theta}(h_n^{-1}), U_{\theta}(h_n^{-1} g_n)) \to 0 \quad \text{as } n \to + \infty
   $$
   where $\dist$ is any metric on $\Fc_{\theta}$ compatible to the standard topology on $\Fc_{\theta}$.
\end{lemma}

\begin{remark} If $\{ g_n^{-1} h_n \}$ is escaping and $\phi(\kappa(g_n)) \le \phi(\kappa(h_n))$ for all $n \ge 1$, then Lemma~\ref{lem:properness} implies that  $\{h_n\}$ is escaping. Then by the definition of transverse groups,
$$
\lim_{n \rightarrow  + \infty} \alpha(\kappa(h_n^{-1})) =+\infty =  \lim_{n \rightarrow  + \infty} \alpha(\kappa(h_n^{-1} g_n)) 
$$
for all $\alpha \in \theta$. Thus $U_{\theta}(h_n^{-1})$ and  $U_{\theta}(h_n^{-1} g_n)$ are both well-defined for $n$ sufficiently large. 

\end{remark}

\begin{proof} 
Suppose not. Then we can find  $\{g_n \}, \{ h_n \} \subset \Ga$ such that $\phi(\kappa(g_n)) \le \phi(\kappa(h_n))$ for all $n \ge 1$,  $\{ g_n^{-1} h_n \}$ is escaping, and
   $$
  \liminf_{n \rightarrow  + \infty} \dist( U_{\theta}(h_n^{-1}), U_{\theta}(h_n^{-1} g_n)) > 0.
   $$
   Passing to a subsequence we can suppose that $U_{\theta}(h_n^{-1}) \rightarrow x$ and $U_{\theta}(h_n^{-1} g_n) \rightarrow y$ with $x \neq y$. Since $x, y \in \La_{\theta}(\Ga)$, we may assume that $\theta$ is symmetric by replacing $\theta$ with $\theta \cup \opp^* \theta$ (see Observation \ref{obs:symmetry theta}), and hence we must have that $x$ and $y$ are transverse due to the $\Psf_{\theta}$-transversality of $\Ga$.

Fix $\alpha \in \theta$. Passing to a subsequence we can suppose that $\frac{\Phi_\alpha(h_n)}{\norm{\Phi_\alpha(h_n)}} \rightarrow T_\alpha$ and $\frac{\Phi_\alpha(h_n^{-1}g_n)}{\norm{\Phi_\alpha(h_n^{-1}g_n)}} \rightarrow S_\alpha$ in $\End(V_\alpha)$. Lemma ~\ref{lem:rank one limits in Valpha}  implies that $\ker T_\alpha = \zeta_\alpha^*(x)$ and ${\rm im} \, S_\alpha = \zeta_\alpha(y)$. Since $x$ is transverse to $y$, Property~\ref{item:boundary maps of Tits repn} implies that  $T_\alpha S_\alpha \neq 0$. Thus 
$$
0 < \norm{T_\alpha S_\alpha} \leq 1.
$$
Then by Property~\ref{item:SVs of Tits repn}, 
\begin{align*}
\lim_{n \rightarrow  + \infty} & N_{\alpha} \omega_\alpha\left(  \kappa(g_n) -\kappa(h_n) - \kappa( h_n^{-1} g_n) \right) = \lim_{n \rightarrow  + \infty} \log \frac{\norm{ \Phi_\alpha( h_n) \Phi_\alpha(h_n^{-1}g_n)}}{\norm{\Phi_\alpha(h_n)}\norm{\Phi_\alpha(h_n^{-1}g_n)}}  \\
& = \log \norm{T_\alpha S_\alpha} 
\end{align*} 
which is finite.
 Since $\alpha \in \theta$ was arbitrary and $\{ \omega_\alpha\}_{\alpha \in \theta}$ is a basis for $\mfa_\theta^*$, there exists $C > 0$ such that 
 $$
-C \leq  \phi\left( \kappa(g_n) -\kappa(h_n) - \kappa( h_n^{-1} g_n) \right)  \leq C
 $$
 for all $n \geq 1$. 
 
 Since the map $\gamma \in \Gamma \mapsto  \phi(\kappa(\gamma)) \in \Rb$ is proper by Lemma \ref{lem:properness}, 
 $$
0 \geq  \lim_{n \rightarrow  + \infty} \phi(\kappa(g_n)) -\phi(\kappa(h_n)) \geq -C +  \lim_{n \rightarrow  + \infty} \phi(\kappa( h_n^{-1} g_n)) = +\infty
$$
and we have a contradiction. 
\end{proof}

\begin{lemma}\label{lem: the claim needed for PS7} For every $R > 0$ there exists $R'> 0$ such that: if $\gamma_1, \gamma_2 \in \Gamma$, $\Oc_{R}^\theta(\gamma_1) \cap \Oc_{R}^\theta(\gamma_2) \neq \emptyset$, and $\phi(\kappa(\ga_1)) \le \phi(\kappa(\ga_2))$, then 
\begin{equation*}
\ga_1^{-1} \Oc_{R}^{\theta}(\ga_2) \subset \Oc_{R'}^{\theta}(\ga_1^{-1} \ga_2).
\end{equation*}
\end{lemma} 

\begin{proof} Suppose not. Then for every $n \geq 1$ we can fix $g_n, h_n \in \Gamma$ and $\xi_n \in \partial_\theta X$ such that $\Oc_{R}^\theta(g_n) \cap \Oc_{R}^\theta(h_n) \neq \emptyset$, $\phi(\kappa(g_n)) \le \phi(\kappa(h_n))$, and
\begin{equation*}
\xi_n \in h_n^{-1} \Oc_{R}^{\theta}(h_n) \smallsetminus h_n^{-1} g_n \Oc_{n}^{\theta}(g_n^{-1} h_n).
\end{equation*}
After passing to a subsequence, we can fix $\alpha \in \theta$  such that 
$$
\omega_{\alpha} \xi_n (h_n^{-1} o)  > \omega_{\alpha} \kappa(h_n) - R \quad \text{and} \quad \omega_{\alpha} \xi_n (h_n^{-1}g_n o)  \le \omega_{\alpha} \kappa(g_n^{-1} h_n) - n
$$
for all $n \geq 1$. 

For each $n$, fix $(T_\beta^{\xi_n})_{\beta \in \theta}$ representing $\xi_n$. Then by Property \ref{item:SVs of Tits repn},
$$
\norm{ \frac{\Phi_{\alpha}(h_n)}{\norm{\Phi_{\alpha}(h_n)}} T_{\alpha}^{\xi_n}}  \ge e^{-N_{\alpha} R} \quad \text{and} \quad 
\norm{ \frac{\Phi_{\alpha}(g_n^{-1}h_n)}{\norm{\Phi_{\alpha}(g_n^{-1} h_n)}} T_{\alpha}^{\xi_n}}  \le e^{-N_{\alpha} n}.
$$
Passing to a subsequence, we can assume 
$$
\frac{\Phi_{\alpha}(h_n)}{\norm{\Phi_{\alpha}(h_n)}} \to S_{\alpha}, \quad \frac{\Phi_{\alpha}(g_n^{-1}h_n)}{\norm{\Phi_{\alpha}(g_n^{-1} h_n)}} \to S_{\alpha}', \quad \text{and} \quad T_\alpha^{\xi_n} \rightarrow T_\alpha
$$
in $\End(V_\alpha)$.  Then  we have
$$
S_{\alpha} T_{\alpha} \neq 0 \quad \text{and} \quad S_{\alpha}' T_{\alpha} = 0.
$$
The second equality implies that $S_\alpha' \notin \SL(V_\alpha)$ and hence $\{ g_n^{-1}h_n\}$ must be escaping. Then by Lemmas \ref{lem:properness} and  \ref{lem:BCZZ converging to the same point} applied to $\theta \cup \opp^* \theta$,  we have 
$$
 \dist( U_{\theta \cup \opp^*\theta}(h_n^{-1}) , U_{\theta \cup \opp^*\theta}(h_n^{-1} g_n) ) \rightarrow 0. 
$$
Then Lemma \ref{lem:rank one limits in Valpha} implies that $\ker S_{\alpha} = \ker S_{\alpha}'$, but this is impossible since $S_{\alpha} T_{\alpha} \neq 0$ and $S_{\alpha}' T_{\alpha} = 0$. 
\end{proof}

We verify the two parts of Property \ref{item:intersecting shadows} separately.

\begin{lemma} \label{lem:transversePS7_1}
For every $R > 0$ there exists $R'> 0$ such that: if $\gamma_1, \gamma_2 \in \Gamma$, $\Oc_{R}^\theta(\gamma_1) \cap \Oc_{R}^\theta(\gamma_2) \neq \emptyset$, and $\phi(\kappa(\ga_1)) \le \phi(\kappa(\ga_2))$, then 
$$
\Oc^\theta_{R}(\gamma_2) \subset \Oc^\theta_{R'}(\gamma_1).
$$
\end{lemma}

\begin{proof} Suppose not. Then for every $n \in \Nb$ we can find $g_n, h_n \in \Gamma$ such that $\phi(\kappa(g_n)) \leq \phi(\kappa(h_n))$, $\Oc^\theta_{R}(g_n) \cap \Oc^\theta_{R}(h_n) \neq \emptyset$,  and 
$$
\Oc_{R}^\theta(h_n) \not\subset \Oc_{n}^\theta(g_n).
$$

\medskip

\noindent \textbf{Case 1:} Assume that $\{g_n^{-1}h_n\}$ is finite. Then, passing to a subsequence, we can suppose that $\gamma:=g_n^{-1}h_n$ for all $n$. 

For each $n \in \Nb$, fix 
$$
\xi_n \in g_n^{-1}\Oc_{R}^\theta(h_n) \smallsetminus g_n^{-1}\Oc_{n}^\theta(g_n).
$$
Passing to a subsequence,  there exist $\alpha \in \theta$ such that 
$$\begin{aligned}
\omega_{\alpha} \ga^{-1} \xi_n (\ga^{-1} g_n^{-1} o) & > \omega_{\alpha} \kappa(g_n \ga) - R \quad \text{and} \quad 
\omega_{\alpha} \xi_n ( g_n^{-1} o )  \le \omega_{\alpha} \kappa(g_n) - n
\end{aligned}
$$
for all $n \in \Nb$. Noting that $\ga^{-1} \xi_n (\ga^{-1} g_n^{-1} o) = \xi_n (g_n^{-1} o) - \xi_n (\ga o)$, the two sequences
$$
\ga^{-1} \xi_n (\ga^{-1} g_n^{-1} o) - \xi_n ( g_n^{-1} o) \quad \text{and} \quad \kappa(g_n \ga) -  \kappa(g_n) 
$$
are uniformly bounded by Lemma \ref{lem:Lipschitz} and Property \ref{item:SVs of Tits repn}. This yields a contradiction.

\medskip

\noindent \textbf{Case 2:} Assume that $\{g_n^{-1}h_n\}$ is escaping. 
By Lemma \ref{lem:properness} and Lemma \ref{lem:BCZZ converging to the same point},  we have 
$$
 \dist( U_{\theta}(h_n^{-1}) , U_{\theta}(h_n^{-1} g_n) ) \rightarrow 0. 
$$

By Lemma~\ref{lem: the claim needed for PS7}, we can fix $R' > 0$ such that 
$$
g_n^{-1} \Oc^\theta_R( h_n) \subset  \Oc_{R'}^\theta(g_n^{-1} h_n)
$$
for all $n \in \Nb$. Since $\Oc_R^{\theta}(g_n) \cap \Oc_R^{\theta}(h_n) \neq \emptyset$ by the hypothesis, we also have $\Oc_{R'}^\theta(g_n^{-1} h_n) \cap g_n^{-1} \Oc_R^{\theta}(g_n) \neq \emptyset$ for all large $n \in \Nb$. Note that ${\rm diam} \Oc_{R'}^\theta(g_n^{-1} h_n) \rightarrow 0$ as $n \to + \infty$ by Lemma \ref{lem:diam decay}. This implies that for any small $r > 0$,
$$
g_n^{-1} \Oc^\theta_R( h_n) \subset  \Oc_{R'}^\theta(g_n^{-1} h_n) \subset \Nc_r( g_n^{-1} \Oc^\theta_{R}(g_n)) 
$$
for all large $n \in \Nb$, where $\Nc_{r}(\cdot)$ denotes the $r$-neighborhood in $\partial_{\theta} X$.

\medskip

\noindent \textbf{Claim:} There exists $r > 0$ such that  
$$
\Nc_r( g_n^{-1} \Oc^\theta_{R}(g_n))  \subset g_n^{-1} \Oc^\theta_{n}(g_n) \quad \text{for all large } n \in \Nb.
$$

\noindent \emph{Proof of Claim:} Suppose not. Then after passing to a subsequence, we have for all $n \in \Nb$ that
$$
\Nc_{1/n}( g_n^{-1} \Oc^\theta_{R}(g_n)) \not \subset g_n^{-1} \Oc^\theta_{n}(g_n).
$$
This implies that after passing to a subsequence there exist $\alpha \in \theta$ and sequences $\{\xi_n \}, \{\xi_n'\} \subset \partial_{\theta} X$ such that
$\xi := \lim_{n \to + \infty} \xi_n = \lim_{n \to + \infty} \xi_n'$ and for all $n \in \Nb$,
$$
\omega_{\alpha} \xi_n (g_n^{-1} o) \le \omega_{\alpha} \kappa(g_n) - n \quad \text{and} \quad \omega_{\alpha} \xi_n'(g_n^{-1} o) > \omega_{\alpha} \kappa(g_n) - R.
$$

For each $n \in \Nb$, let $\{p_{n, k} \},\{p_{n, k}'\} \subset \Gsf$ be sequences such that $p_{n, k} o \to \xi_n$ and $p_{n, k}' o \to \xi_n'$ in $\overline{X}^{\theta}$ as $k \to + \infty$.
Then by Lemma \ref{lem:Titsexpress} and Property \ref{item:SVs of Tits repn}, we have 
$$
\lim_{k \to + \infty} \norm{ \frac{\Phi_{\alpha}(g_n)}{\norm{\Phi_{\alpha}(g_n)}} \frac{\Phi_{\alpha}(p_{n, k})}{\norm{\Phi_{\alpha}(p_{n, k})}}} \le e^{-N_{\alpha} n} \quad \text{and} \quad \lim_{k \to + \infty} \norm{ \frac{\Phi_{\alpha}(g_n)}{\norm{\Phi_{\alpha}(g_n)}} \frac{\Phi_{\alpha}(p_{n, k}')}{\norm{\Phi_{\alpha}(p_{n, k}')}}} > e^{-N_{\alpha} R}
$$
Hence, for each $n \in \Nb$, we can choose $p_n := p_{n, k_n}$ and $p_n' := p_{n, k_n'}$ for some $k_n, k_n' \in \Nb$ so that $p_n o \to \xi$ and $p_n' o \to \xi$ in $\overline{X}^{\theta}$, and
$$
\lim_{n \to + \infty} \norm{ \frac{\Phi_{\alpha}(g_n)}{\norm{\Phi_{\alpha}(g_n)}} \frac{\Phi_{\alpha}(p_{n})}{\norm{\Phi_{\alpha}(p_{n})}}} = 0 \quad \text{and} \quad \lim_{n \to + \infty} \norm{ \frac{\Phi_{\alpha}(g_n)}{\norm{\Phi_{\alpha}(g_n)}} \frac{\Phi_{\alpha}(p_{n}')}{\norm{\Phi_{\alpha}(p_{n}')}}} > \frac{e^{-N_{\alpha} R}}{2}.
$$

After passing to a subsequence, we may assume
$$
\frac{\Phi_{\alpha}(g_n)}{\norm{\Phi_{\alpha}(g_n)}} \to S, \quad \frac{\Phi_{\alpha}(p_{n})}{\norm{\Phi_{\alpha}(p_{n})}} \to P, \quad \text{and} \quad \frac{\Phi_{\alpha}(p_{n}')}{\norm{\Phi_{\alpha}(p_{n}')}} \to P' \quad \text{in } \End(V_{\alpha}).
$$
Then it follows from $SP = 0$ that $\{g_n\}$ is escaping, and hence $\alpha(\kappa(g_n)) \to + \infty$. For a sequence $\{\ell_n\} \subset \Ksf$ with $g_n \in \Ksf \Asf^+ \ell_n$ for all $n \in \Nb$, we can assume $\ell_n \to \ell \in \Ksf$ by passing to a subsequence. Then by Lemma \ref{lem:rank one limits in Valpha} and Property \ref{item:boundary maps of Tits repn}, we have $\Phi_{\alpha}(\ell^{-1}) V_{\alpha}^- = \ker S$ and hence 
$$
{\rm im } \, P \subset \Phi_{\alpha}(\ell^{-1}) V_{\alpha}^- \quad \text{and} \quad {\rm im } \, P' \not\subset \Phi_{\alpha}(\ell^{-1}) V_{\alpha}^-.
$$
Since $\norm{P} \le 1$ and ${\rm im } \, P \subset \Phi_{\alpha}(\ell^{-1}) V_{\alpha}^-$, we have for any $a \in \Asf^+$ that 
$$
\norm{ \Phi_{\alpha}(a \ell) P} \le e^{(N_{\alpha} \omega_{\alpha} - \alpha)(\kappa(a))}
$$
by Properties \ref{item:SVs of Tits repn} and \ref{item:splitting for Tits repn}.
Since ${\rm im } \, P' \not\subset \Phi_{\alpha}(\ell^{-1}) V_{\alpha}^-$, there exists a unit $v \in V_{\alpha}$ such that $\Phi_{\alpha}(\ell)P' v = u + w$ for some $u \in V_{\alpha}^+$ and $w \in V_{\alpha}^-$ with $u \neq 0$. Hence, we similarly have for any $a \in \Asf^+$ that 
$$
\norm{ \Phi_{\alpha}(a \ell) P'} \ge e^{N_{\alpha} \omega_{\alpha}(\kappa(a))} \norm{u} - e^{(N_{\alpha} \omega_{\alpha} - \alpha)(\kappa(a))} \norm{w}.
$$

On the other hand, since $\{ p_n o\}, \{p_n' o\} \subset \overline{X}^{\theta}$ converges to the same limit $\xi$, it follows from Lemma \ref{lem:Titsexpress} that for any $a \in \Asf^+$,
$$
\norm{ \Phi_{\alpha}(a \ell) P} = \norm{ \Phi_{\alpha}(a \ell) P'},
$$
and therefore 
$$
e^{(N_{\alpha} \omega_{\alpha} - \alpha)(\kappa(a))} \ge e^{N_{\alpha} \omega_{\alpha}(\kappa(a))} \norm{u} - e^{(N_{\alpha} \omega_{\alpha} - \alpha)(\kappa(a))} \norm{w}.
$$
This implies
$$
e^{-\alpha(\kappa(a))} (1 + \norm{w}) \ge \norm{u}.
$$
Since this holds for any $a \in \Asf^+$ and $\norm{u} > 0$, this is a contradiction and the claim follows. \hfill $\blacktriangleleft$ 

\medskip

Now by the claim, we have
\begin{align*}
g_n^{-1} \Oc^\theta_R( h_n) \subset  \Oc_{R'}^\theta(g_n^{-1} h_n) \subset \Nc_r( g_n^{-1} \Oc^\theta_{R}(g_n))  \subset g_n^{-1} \Oc^\theta_{n}(g_n)
\end{align*}
for $n$ large, which is a contradiction. This finishes the proof.
\end{proof} 

Now we complete the proof of Property \ref{item:intersecting shadows}, and hence of Theorem \ref{thm:transverse well-behaved}, by showing the following. 

\begin{lemma} \label{lem:transversePS7_2}
For every $R > 0$ there exists $C > 0$ such that: if $\gamma_1, \gamma_2 \in \Gamma$, $\Oc_{R}^\theta(\gamma_1) \cap \Oc_{R}^\theta(\gamma_2) \neq \emptyset$, and $\phi(\kappa(\ga_1)) \le \phi(\kappa(\ga_2))$, then
$$
\abs{ \phi\left( \kappa (\ga_2) \right) - \left(  \phi \left( \kappa (\ga_1) \right) + \phi \left( \kappa (\ga_1^{-1}\ga_2) \right)\right) } \le C.
$$

\end{lemma} 

\begin{proof}
Suppose not. Then for every $n \in \Nb$ we can find $g_n, h_n \in \Gamma$ such that $ \phi( \kappa(g_n) )\leq \phi(\kappa(h_n))$, $\Oc^\theta_{R}(g_n) \cap \Oc^\theta_{R}(h_n) \neq \emptyset$,  and 
$$
\abs{ \phi( \kappa (h_n) ) - \left( \phi(\kappa (g_n)) + \phi( \kappa (g_n^{-1}h_n)) \right) } > n.
$$
First, it is easy to see that $\{g_n\}$ and $\{g_n^{-1} h_n\}$ are escaping sequences. Then we have $\phi(\kappa(h_n)) \ge \phi(\kappa(g_n)) \to + \infty$ by Lemma \ref{lem:properness}. In particular, $\{h_n\}$ is an escaping sequence.

For each $n \in \Nb$, fix $\xi_n \in \Oc^\theta_{R}(g_n) \cap \Oc^\theta_{R}(h_n) $. Then by Lemma~\ref{lem: the claim needed for PS7},  there exists $R' > 0$ such that $\xi_n \in g_n \Oc_{R'}^{\theta}(g_n^{-1} h_n)$, and hence 
$$
h_n^{-1}  \xi_n  \in (g_n^{-1}h_n)^{-1} \Oc_{R'}^{\theta}(g_n^{-1} h_n).
$$
Therefore, together with Lemma \ref{lem:Titsexpress}, we have for each $\alpha \in \theta$ that 
$$
\begin{aligned}
\omega_{\alpha}  \kappa (g_n) & \ge \omega_{\alpha} g_n^{-1} \xi_n (g_n^{-1} o) \ge \omega_{\alpha} \kappa(g_n) - R \\
\omega_{\alpha} \kappa (h_n) & \ge \omega_{\alpha} h_n^{-1} \xi_n (h_n^{-1} o) \ge \omega_{\alpha} \kappa(h_n) - R \\
\omega_{\alpha} \kappa (g_n^{-1}h_n) & \ge \omega_{\alpha} h_n^{-1} \xi_n ((g_n^{-1}h_n)^{-1} o) \ge \omega_{\alpha} \kappa(g_n^{-1} h_n) - R'.
\end{aligned}
$$
Notice that 
\begin{align*}
h_n^{-1} \xi_n ((g_n^{-1}h_n)^{-1} o) & = h_n^{-1} \xi_n ( h_n^{-1} g_n o) = \xi_n( h_n h_n^{-1} g_n o) - \xi_n(h_n o) \\
& = h_n^{-1} \xi_n (h_n^{-1} o)-g_n^{-1} \xi_n (g_n^{-1} o).
\end{align*}
Hence  for each $\alpha \in \theta$, 
$$
-R - R' \le \omega_{\alpha} \kappa (h_n) - \omega_{\alpha} \kappa (g_n) - \omega_{\alpha} \kappa (g_n^{-1}h_n) \le R.
$$
Since $\phi$ is a linear combination of $\{ \omega_{\alpha} : \alpha \in \theta \}$, this is a contradiction.
\end{proof}

\subsection{Proof of Theorem~\ref{thm:uniqueness for transverse}} For Patterson--Sullivan-measures supported on $\Lambda_\theta(\Gamma)$, uniqueness and hence ergodicity was established in ~\cite[Corollaries 12.1, 12.2]{CZZ2024}. Using this result and Proposition \ref{prop:concon limit set for general groups} it suffices to fix a $(\Ga, \phi, \delta^{\phi}(\Ga))$-Patterson--Sullivan measure $\mu$ on $\partial_\theta X$ and show that $\mu(\Lambda_\theta^{\rm con}(\Gamma)) = 1$.

Theorem~\ref{thm:div => pos meas conical} implies that  $\mu(\Lambda_\theta^{\rm con}(\Gamma)) > 0$ and Proposition~\ref{prop:concon limit set for general groups}  implies that $\Lambda_\theta^{\rm con}(\Gamma)$ is $\Gamma$-invariant. Suppose for a contradiction that $\mu(\Lambda_\theta^{\rm con}(\Gamma)) < 1$. Then 
$$
\mu' (\cdot): = \frac{1}{\mu( \partial_\theta X \smallsetminus \Lambda_\theta^{\rm con}(\Gamma))} \mu\big( \cdot \cap (  \partial_\theta X \smallsetminus \Lambda_\theta^{\rm con}(\Gamma))\big)
$$
is a $(\Ga, \phi, \delta^{\phi}(\Ga))$-Patterson--Sullivan measure $\mu$ on $\partial_\theta X$. So Theorem~\ref{thm:div => pos meas conical} applied to $\mu'$ implies that  $\mu'(\Lambda_\theta^{\rm con}(\Gamma)) > 0$ which is impossible. Thus $\mu(\Lambda_\theta^{\rm con}(\Gamma)) = 1$ and hence uniqueness and ergodicity follow from ~\cite[Corollaries 12.1, 12.2]{CZZ2024}. The ``in particular'' part is due to Proposition \ref{prop:concon limit set for transverse groups}.
\qed

\bibliographystyle{alpha}
\bibliography{geom}

\end{document}